\documentclass[12pt]{amsart}
\pdfoutput=1

\newcommand\myChanged[1]{#1}
\newcommand\myRemoved[1]{}

\usepackage[headings]{fullpage}
\usepackage{amssymb,epic,eepic,epsfig}
\usepackage{epstopdf}
\usepackage{graphicx}
\usepackage{texdraw}
\usepackage{url}
\usepackage[vlined,boxed,norelsize]{algorithm2e}
\usepackage[all]{xy}
\usepackage{float}   

\usepackage[bookmarks=true,%
    colorlinks=true,%
    linkcolor=blue,%
    citecolor=blue,%
    filecolor=blue,%
    menucolor=blue,%
    urlcolor=blue,%
    breaklinks=true]{hyperref}

\newtheorem{theorem}{Theorem}[section]
\theoremstyle{definition}

\newtheorem{lemma}[theorem]{Lemma}
\newtheorem{definition}[theorem]{Definition}
\newtheorem{remark}[theorem]{Remark}
\newtheorem{corollary}[theorem]{Corollary}
\newtheorem{conjecture}[theorem]{Conjecture}

\newtheorem{example}[theorem]{Example}

\def\BZ{\mathbb Z}

\def\BQ{\mathbb Q}

\def\BC{\mathbb C}
\def\BH{\mathbb H}

\def\otet{\mathrm{otet}}
\def\ntet{\mathrm{ntet}}
\def\PGL{\mathrm{PGL}}
\def\PSL{\mathrm{PSL}}

\def\tman{tetrahedral manifold}
\def\Tman{Tetrahedral manifold}
\def\ctt{combinatorial tetrahedral tessellation}

\def\mg{\textcolor{red}{$\diamondsuit$}}

\newcommand{\myComment}[1]{}

\begin{document}


\title[A census of tetrahedral hyperbolic manifolds]{
A census of tetrahedral hyperbolic manifolds}
\author[Fominykh]{Evgeny Fominykh}
\address{Laboratory of Quantum Topology \\
         Chelyabinsk State University \\ 
         Chelyabinsk, 454001, Russia, and Institute of Mathematics 
         and Mechanics \\ 
         Ekaterinburg, 620990, Russia}
\email{efominykh@gmail.com}
\author[Garoufalidis]{Stavros Garoufalidis}
\address{School of Mathematics \\
         Georgia Institute of Technology \\
         Atlanta, GA 30332-0160, USA \newline 
         {\tt \url{http://www.math.gatech.edu/~stavros}}}
\email{stavros@math.gatech.edu}
\author[Goerner]{Matthias Goerner} 
\address{Pixar Animation Studios\\ 
         1200 Park Avenue\\ 
         Emeryville, CA 94608, USA \newline
         {\tt \url{http://www.unhyperbolic.org/}}}
\email{enischte@gmail.com}
\author[Tarkaev]{Vladimir Tarkaev}
\address{Laboratory of Quantum Topology \\
         Chelyabinsk State University \\ 
         Chelyabinsk, 454001, Russia}
\email{trk@csu.ru}
\author[Vesnin]{Andrei Vesnin}
\address{Sobolev Institute of Mathematics \\ 
         Siberian Branch of the Russian Academy of Sciences \\
         Novosibirsk, 630090, Russia \newline
         {\tt \url{http://www.math.nsc.ru/~vesnin}}}
\email{vesnin@math.nsc.ru}
\thanks{
1991 {\em Mathematics Classification.} Primary 57N10. Secondary 57M25.
}

\date{October 6, 2015}


\begin{abstract}
We call a cusped hyperbolic 3-manifold {\em tetrahedral} if it can be 
decomposed into regular ideal tetrahedra. Following an earlier 
publication by three of the authors, we give a census of all tetrahedral 
manifolds and all of their combinatorial tetrahedral tessellations
with at most 25 (orientable case) and 21 (non-orientable case) tetrahedra. 
Our isometry classification uses certified canonical cell decompositions 
(based on work by Dunfield, Hoffman, Licata) and isomorphism signatures 
(an improvement of dehydration sequences by Burton). The tetrahedral census 
comes in \texttt{Regina} as well as \texttt{SnapPy} format, and we 
illustrate its features. 
\end{abstract}

\maketitle

\tableofcontents

\section{Introduction}
\label{sec.intro}

\subsection{\Tman s}
\label{sub.goal}

We call a cusped hyperbolic 3-manifold {\em tetrahedral} if it can be 
decomposed into regular ideal tetrahedra. The combinatorial data of this 
decomposition is captured in the {\em \ctt} which can be defined simply as an 
ideal triangulation where all edges have order 6. By Mostow rigidity, a 
combinatorial tetrahedral tessellation \myRemoved{removed uniquely} 
determines a \tman. However, there might be several non-isomorphic 
\myChanged{(i.e., not related by just relabeling tetrahedra and vertices)} 
\ctt s yielding the same tetrahedral 
manifold. That is why we introduce the two terms \tman{} and \ctt{} to 
distinguish whether we regard isometric or combinatorially isomorphic objects 
as equivalent.

The \tman{} were also called {\em maximum volume} in~\cite{A,VMF,VTF1, VTF2} 
because they are precisely the ones with maximal volume among all hyperbolic 
manifolds with a fixed number of tetrahedra. Thus, they also appear at the 
trailing ends of the \texttt{SnapPy} \cite{SnapPy} census manifolds sharing 
the same 
letter\footnote{The case of the letter \texttt{m} is exceptional because it 
spans several number of tetrahedra for purely historic reasons.} (e.g., 
\texttt{m405} to \texttt{m412}, \texttt{s955} to 
\texttt{s961}, \texttt{v3551}, \texttt{t12833} to \texttt{t12845}, 
\texttt{o9\_44249}). Moreover, the number of tetrahedra and the Matveev 
complexity~\cite{Matveev} also coincides for these manifolds.

The census of \tman s illustrates a number of phenomena of arithmetic
hyperbolic manifolds including symmetries visible
in the canonical cell decomposition \myChanged{but hidden by} the \ctt. 
In particular, the 
canonical cell decomposition might have non-tetrahedral cells.

Several manifolds that have played a key role in the development of hyperbolic
geometry are tetrahedral, e.g., the complements of the figure-eight knot, the
minimally twisted 5-chain link (which conjecturally is also the minimum 
volume orientable hyperbolic manifold with 5 cusps) and the Thurston 
congruence link.
The last two have the special property that their \ctt{} is maximally 
symmetric, i.e., any tetrahedron can be taken to any other tetrahedron in 
every orientation-preserving configuration via a combinatorial isomorphism. 
One of the authors has classified link complements with this special 
property in previous work \cite{Goerner}.

We also construct several new links with tetrahedral complement.

\subsection{Our results and methods} 
\label{sub.results}

Our main goals   (see \cite{Goerner:tetcensus} for the data) are the 
creation of
\begin{enumerate}
\renewcommand{\theenumi}{\alph{enumi}}
\item \label{goal:ctt}
The census of \ctt s up to 25 (orientable case), respectively, 21 
(non-orientable case) tetrahedra. 
\item \label{goal:homeoCtt}
The grouping by isometry type and the corresponding canonical cell 
decompositions.\\
We ship this as a \texttt{Regina} \cite{Regina} file containing 
triangulations in a 
hierarchy reflecting the grouping.
\item \label{goal:tman}
The corresponding census of \tman s.\\
We ship this as a \texttt{SnapPy} census containing a representative 
triangulation for each isometry type.
This census can be used just like any other \texttt{SnapPy} census.
\item \label{goal:morphisms}
The list of covering maps between the \ctt s.
\end{enumerate}

For~(\ref{goal:ctt}), we use a new approach differing from the traditional 
one that starts by enumerating 4-valent graphs used first by 
Callahan-Hildebrand-Weeks~\cite{CHW} \myChanged{or variations of the 
traditional approach such as by Burton and Pettersson \cite{Burton:edges}}. 
The advantage of our new approach is 
that it scales to a substantially higher number of tetrahedra because it 
allows for early pruning of triangulations with edges of wrong order. We also 
deploy isomorphism signatures to avoid recounting combinatorially isomorphic 
triangulations. Recall that the isomorphism signature is an improvement 
by Burton~\cite{burton:encode}  of the (non-canonical) dehydration sequences. 
It is a complete invariant of the combinatorial isomorphism type of a 
triangulation. Algorithms~~1 and~2 used for the enumeration of \ctt s are 
described  in Section~\ref{section2}.  
Isomorphism signatures of orientable  \ctt s with at most seven tetrahedra 
are presented in Table~\ref{table:isomorph}. 

For~(\ref{goal:homeoCtt}), we use a new invariant we call the 
{\em isometry signature} (see Section~\ref{section3}). It is a complete 
invariant of the isometry type 
of a cusped hyperbolic 3-manifold. It is defined as the isomorphism 
signature of the canonical retriangulation of the canonical cell 
decomposition ~\cite{EP}. To compute it, we use exact arithmetic to certify 
the canonical cell decomposition even when the cells are not tetrahedral,
expanding on work by Dunfield, Hoffman, Licata~\cite{DHL}.

For~(\ref{goal:morphisms}), we wrote a script that finds combinatorial 
homomorphisms from a triangulation to another triangulation.

Several of the techniques here are new and can be generalized: The 
{\em isometry
signature} is an invariant that is defined for any finite-volume cusped 
hyperbolic
3-manifolds. It is a complete isometry invariant (and thus by Mostow 
rigidity a complete homotopy
invariant) that can be effectively computed and,
in general, be certified whenever the manifold is orientable and the 
canonical cell
decomposition contains only tetrahedral cells using 
{\tt hikmot}~\cite{hikmot}.
We also provide an improvement of the code provided in \cite{DHL} 
to certify
canonical triangulations that is simpler and generalizes to any number 
of cusps.

\myRemoved{After reading Burton/Petterson's paper more, I realized that 
their approach shares more with the traditional one than with the one 
taken here.}

Applying the above discussed methods we obtain the following result.

\begin{theorem}
The number of  \ctt s and \tman s 
up to 25 tetrahedra for orientable manifolds and up to 21 tetrahedra for 
non-orientable manifolds are listed in Table~\ref{table:censusNumbers}.  
\end{theorem}

\begin{table}
\caption{Number of triangulations in the census.\label{table:censusNumbers}}
\bigskip
\begin{tabular}{r||r|r||r|r||r}
& \multicolumn{2}{|c||}{combinatorial} & \multicolumn{2}{|c||}{tetrahedral} 
& \multicolumn{1}{c}{homology} \\
& \multicolumn{2}{|c||}{tet. tessellations} & \multicolumn{2}{|c||}{manifolds} 
& \multicolumn{1}{c}{links} \\
Tetrahedra & orientable & non-or. & orientable & non-or. & \\ \hline \hline
1 & 0 & 1 & 0 & 1 & 0 \\ \hline
2 & 2 & 2 & 2 & 1 & 1 \\ \hline
3 & 0 & 1 & 0 & 1 & 0 \\ \hline
4 & 4 & 4 & 4 & 2 & 2 \\ \hline
5 & 2 & 12 & 2 & 8 & 0 \\ \hline
6 & 7 & 14 & 7 & 10 & 0 \\ \hline
7 & 1 & 1 & 1 & 1 & 0 \\ \hline
8 & 14 & 10 & 13 & 6 & 5 \\ \hline
9 & 1 & 6 & 1 & 6 & 0 \\ \hline
10 & 57 & 286 & 47 & 197 & 12 \\ \hline
11 & 0 & 17 & 0 & 17 & 0 \\ \hline
12 & 50 & 117 & 47 & 80 & 7 \\ \hline
13 & 3 &8 & 3 & 8 & 0 \\ \hline
14 & 58 &134 & 58 & 113 & 25 \\ \hline
15 & 91 & 975 & 81 & 822 & 0 \\ \hline
16 & 102 & 175 & 96 & 142 & 32 \\ \hline
17 & 8 & 52 & 8 & 52 & 0 \\ \hline
18 & 213 & 1118 & 199 & 810 & 66 \\ \hline
19 & 25 & 326 & 25 & 326 & 0 \\ \hline
20 & 1886 & 26320 & 1684 & 22340 & 209 \\ \hline
21 & 31 & 251 & 31 & 251 & 0 \\ \hline
22 & 390 &- & 381 & - & 148 \\ \hline
23 & 58 &- & 58 & - & 0 \\ \hline
24 & 1544 & - & 1465 & - & 378 \\ \hline
25 & 7563 & - & 7367 & - & 0
\end{tabular}
\end{table}

All  \ctt s and \tman s indicated in Table~\ref{table:censusNumbers} are 
enumerated in supplement files available in~\cite{Goerner:tetcensus}.  

Knots and links with tetrahedral complement are shown in 
Figures~\ref{f.tetlinks}, \ref{f.tetlinks12} and \ref{fig.remarkable}.

\subsection{Features of the tetrahedral census}

Properties of tetrahedral manifolds that make them 
interesting to study include:

\begin{itemize}
\item
\myChanged{The tetrahedral manifolds are arithmetic as they are a proper 
subset of the commensurability class of figure-eight knot complement, 
closed under finite coverings, see Section~\ref{sec:Margulis}.}

\myComment{
All tetrahedral manifolds are arithmetic and commensurable with the 
figure-eight
knot complement, see Lemma~\ref{lem.arithmetic}. The converse is not true: 
In Remark~\ref{rem.nottrue}, we list examples of 
manifolds commensurable with the figure-eight knot which are not tetrahedral.

\item
The tetrahedral census is closed under finite (possibly irregular) coverings. 
We investigate the resulting category in Section~\ref{section:cttCat}.
}

\item
The tetrahedral manifolds are exactly those with maximal volume among all
cusped hyperbolic manifolds with a fixed number of tetrahedra.
\item 
Their Matveev complexity equals the number of regular ideal tetrahedra.

\item
Many \ctt s \myChanged{hide symmetries, i.e, there are isometries of the 
corresponding \tman{} that are not induced from a combinatorial 
isomorphism of the \ctt.}

\item 
A substantial fraction of tetrahedral manifolds are link complements.
\end{itemize}

\section{The enumeration of \ctt s} 
\label{section2} 

\begin{algorithm}
\SetKwProg{Proc}{Procedure}{}{end}%
\SetKwProg{Fn}{Function}{}{end}%
\SetKwFunction{Search}{RecursiveFind}%
\SetKwFunction{Main}{FindAllTetrahedralTessellations}%
\Fn{\Main{integer max, bool orientable}}{
\KwResult{Returns all (non-)orientable tetrahedral tessellations up to 
combinatorial isomorphism with at most \ArgSty{max} tetrahedra.}
~

result $\leftarrow \{\}$ \tcc*{resulting triangulations}
already\_seen $\leftarrow \{\}$ \tcc*{isomorphism signatures encountered 
earlier}
~

\Proc{\Search{Triangulation $t$}}{
\KwResult{Searches all triangulations obtained from $t$ by gluing faces 
or adding tetrahedra.}
~

\tcc{Close order 6 edges and reject unsuitable triangulations}
\If{ \FuncSty{FixEdges(}$t$\FuncSty{)} = ``valid'' }{
\tcc{Skip triangulations already seen earlier}
\If{ \FuncSty{isomorphismSignature(}$t$\FuncSty{)} $\not\in$ already\_seen}{
already\_seen $\leftarrow$ already\_seen $\cup$ 
\{\FuncSty{isomorphismSignature(}$t$\FuncSty{)}\}\;

\If{ t has no open faces }{
\tcc{t orientable by construction if orientable = true}
\If{t is non-orientable or orientable = true}{
result $\leftarrow$ result $\cup ~\{t\}$\;
}
}
\Else{
\tcc{This choice results in faster enumeration}
choose an open face $F_1$=(tetrahedron, $f_1$) of $t$ adjacent to an 
open edge of highest order\;
\If{t has less than max tetrahedra}{
\FuncSty{ResursiveFind(}\ArgSty{t} with a new tetrahedron glued to $F_1$ 
via an odd permutation\FuncSty{)}
}

\For{each open face $F_2 \not=F_1$ of t}{
\For{each $p \in S_4$}{
\If{$p(f_1)=f_2$}{
\If{p is odd or orientable = false}{
\FuncSty{RecursiveFind(}$t$ with $F_1$ glued to $F_2$ via $p$\FuncSty{)}\;
}
}
}
}
}
}
}
}
\FuncSty{RecursiveFind(}triangulation with one unglued tetrahedron\FuncSty{)}\;
\KwRet{result}
}
\caption{The main function to enumerate all tetrahedral tessellations.}
\label{algo:main}
\end{algorithm}

\begin{algorithm}
\SetKwProg{Fn}{Function}{}{end}
\SetKwFunction{FRecurs}{FixEdges}%
\Fn{\FRecurs{Triangulation t}}{
\KwResult{t is modified in place. Returns``valid'' or ``invalid''.}
\While{t has open edge e of order 6}{
close edge $e$\;
}
\KwRet{``valid'' if every edge e
\begin{itemize}\vspace{-1mm}\setlength{\itemsep}{0pt}\item has order $<6$ 
(if open) or $=6$ (if closed) and\item has no projective plane as vertex 
link.\end{itemize}}
}
\caption{A helper function closing order 6 edges and rejecting 
triangulations which cannot result in tetrahedral tessellations.}
\label{algo:fixEdges}
\end{algorithm}

We use Algorithm~\ref{algo:main} to enumerate the \ctt s. The input is the 
maximal number of tetrahedra to be considered and a flag indicating whether we
wish to enumerate the orientable or the non-orientable tessellations. The 
result is a set of ideal triangulations where each edge has order 6 
resulting in manifolds of the desired orientability.

As pointed out in the introduction our algorithm differs from the 
traditional approach:
we recursively try all possible ways open faces can be face-paired without
enumerating 4-valent graphs first. This will, of course, result in many 
duplicates, so
we keep a set of isomorphism signatures (see \cite{burton:encode}) of 
previously encountered triangulations
around to prevent recounting. Recall that an isomorphism signature is, 
unlike a dehydration sequence, a complete invariant of the combinatorial 
isomorphism type of a triangulation.

The advantage of this approach is that we can insert a procedure that can 
prune the
search space early on. In our case, this procedure is given in
Algorithm~\ref{algo:fixEdges} and rejects ideal triangulations where edges 
have the wrong
order. It also rejects ideal triangulations with non-manifold topology. 
These can occur
when the tetrahedra around an edge cannot be oriented consistently and the 
vertex link of
the center of the edge becomes a projective plane $\mathbb{RP}^2$.

The algorithm has been implemented using \texttt{Regina} and we briefly 
recall how a triangulation is presented. The vertices of each tetrahedron 
are indexed 0, 1, 2, 3 and the faces are indexed by the number of the vertex 
opposite  to it. Triangulations in intermediate stages will have unpaired 
faces. We call a face open if it is unpaired, otherwise closed.
A triangulation consists of a number of tetrahedra and for each 
tetrahedron $T_1$ and each face index $f_1=0, ..., 3$, we store two pieces 
of data to encode whether and how the face $F_1=(T_1,f_1)$ is glued to 
another face $F_2=(T_2,f_2)$ with face index $f_2$ of another 
(not necessarily distinct) tetrahedron $T_2$:
\begin{enumerate}
\item A pointer to $T_2$. If $F_1$ is an open face, this pointer is null. 
\item An element $p\in S_4$ such that $p(f_1) = f_2$ and the vertex 
$i\not =f_1$ of $T_1$ is glued to $p(i)$ of $T_2$.
\end{enumerate}

The face pairings implicitly determine edge classes. We call such an edge 
open if it is adjacent to an open face (necessarily so exactly two) and 
otherwise closed. Closing an open edge means gluing the two open adjacent 
faces by the suitable permutation.

The source for the implementation is in  
\texttt{src/genIsomoSigsOfTetrahedralTessellations.cpp}. See 
\cite[\texttt{data/}]{Goerner:tetcensus} for isomorphism signatures of all 
\ctt s from Table~\ref{table:censusNumbers} and Table~\ref{table:isomorph} 
for orientable tessellations with $n \leq 7$ tetrahedra. Also names of 
manifolds in the census are presented (see Section~\ref{section4}). 

\begin{table}[h]
\caption{Isomorphism signatures for all orientable \ctt s with $n \leq 7$ 
tetrahedra. \label{table:isomorph}}
\bigskip
\begin{tabular}{r|l|l||r|l|l}
n & Signatures &Name & n & Signatures & Name \cr \hline
2 & \texttt{cPcbbbdxm} & $\otet02_{0000}$ & 6 & \texttt{gLLPQccdfeefqjsqqjj} & $\otet06_{0000}$ \cr \hline 
2 & \texttt{cPcbbbiht} & $\otet02_{0001}$& 6 & \texttt{gLLPQccdfeffqjsqqsj} & $\otet06_{0001}$ \cr \hline 
4 & \texttt{eLMkbbdddemdxi} & $\otet04_{0000}$ &  6 & \texttt{gLLPQceefeffpupuupa} & $\otet06_{0002}$ \cr \hline 
4 & \texttt{eLMkbcddddedde} & $\otet04_{0001}$ & 6 & \texttt{gLMzQbcdefffhxqqxha} & $\otet06_{0003}$ \cr \hline
4 & \texttt{eLMkbcdddhxqdu} & $\otet04_{0002}$ & 6 & \texttt{gLMzQbcdefffhxqqxxq} & $\otet06_{0004}$ \cr \hline 
4 & \texttt{eLMkbcdddhxqlm} & $\otet04_{0003}$ & 6 & \texttt{gLvQQadfedefjqqasjj} & $\otet06_{0005}$ \cr \hline 
5 & \texttt{fLLQcbcedeeloxset} & $\otet05_{0000}$ & 6 & \texttt{gLvQQbefeeffedimipt} & $\otet06_{0006}$ \cr \hline 
5 & \texttt{fLLQcbdeedemnamjp} & $\otet05_{0001}$ & 7 & \texttt{hLvAQkadfdgggfjxqnjnbw} & $\otet07_{0000}$ \cr 
\end{tabular}
\end{table}

\section{The isometry signature} 
\label{section3} 

In the previous section, we enumerated all \ctt s with a given maximal 
number of tetrahedra up to combinatorial isomorphism. In the next step,
we want to find the equivalence classes of those \ctt s yielding the same 
\tman{} up to isometry.

We do this by grouping \ctt s by their {\em isometry signature} which we 
define, compute and certify in this section. To summarize, the isometry 
signature is the isomorphism signature of the canonical retriangulation 
of the canonical cell decomposition. If, however, the canonical cell 
decomposition has simplices as cells, we short-circuit and just use the 
isomorphism signature of the canonical cell decomposition itself.
We can certify the isometry signature by using exact computations to determine
which faces in the proto-canonical triangulation are transparent.

The code implementing the certified canonical retriangulation can be found in
\texttt{src/canonical\_o3.py}. The code to group (and name) the \ctt s by 
isometry signature
is in 
\texttt{src/identifyAndNameIsometricIsomoSigsOfTetrahedralTessellations.py}.

\subsection{Definition} \label{sec:defCanonicalCell}

Recall that the hyperboloid model of 3-dimensional hyperbolic space $\BH^3$ 
in (3+1)-Minkowski space (with inner product defined by 
$\langle x, y\rangle = x_0y_0 + x_1y_1+x_2y_2-x_3y_3$) is given by 
$$
S^+=\left\{x=(x_0,...,x_3) \,\, | \,\, 
x_3>0, \quad \langle x,x\rangle = -1\right\}.
$$

For a cusped hyperbolic manifold $M$, choose a horotorus cusp neighborhood 
of the same volume for each cusp. Lift $M$ and the cusp neighborhoods to 
$\BH^3\cong S^+$. The cusp neighborhoods lift to a $\pi_1(M)$-invariant set 
of horoballs. For each horoball $B\subset S^+$, there is a dual 
vector $v_B$ that is light-like (i.e., $\langle v_B, v_B \rangle =0$) and 
such that $w\in B \Leftrightarrow \langle v_B, w\rangle > -1$. 
The boundary of the convex hull of all $v_B$ has polygonal faces.

\begin{definition}
The {\em canonical cell decomposition} of $M$ is given by the radial 
projection of the polygonal faces of the boundary of the convex hull of 
all $v_B$ onto $S^+$.
\end{definition}

The canonical cell decomposition was introduced by Epstein and 
Penner~\cite{EP}. It does not depend on a particular choice of cusp 
neighborhoods as long as they all have the same volume, or equivalently, same
area. 

\begin{definition}
A triangulation which is obtained by subdividing the cells of the canonical
cell decomposition and inserting (if necessary) flat tetrahedra is called 
a {\em proto-canonical
triangulation}. If it contains no flat tetrahedra, i.e., all tetrahedra are 
positively
oriented, it is called a {\em geometric proto-canonical triangulation}. 
\end{definition}

The result of calling \texttt{canonize} on a \texttt{SnapPy} manifold is a 
proto-canonical triangulation. If the canonical cell decomposition has 
cells which are not ideal tetrahedra (non-regular or regular), there might 
be more than 
one proto-canonical triangulation of the same manifold. A face of a 
proto-canonical triangulation which is part of a 2-cell of the canonical 
cell decomposition is called {\em opaque}. Otherwise, a face is called 
{\em transparent}.

\begin{definition}
Consider a 2-cell in the canonical cell decomposition which is an $n$-gon. 
Pick
the suspension of such an $n$-gon by the centers of the two neighboring 
3-cells.
These suspensions over all 2-cells form a decomposition of $M$ into 
topological diamonds. 
Each diamond
can be split into $n$ tetrahedra along its central axis. The result is 
called the {\em canonical retriangulation}.
\end{definition}

The canonical retriangulation carries exactly the same information as the
canonical cell decomposition (just packaged as a triangulation) and thus
only depends on (and uniquely determines) 
the isometry type of the manifold. \texttt{SnapPy} uses it internally to 
compute, for example, the symmetry group of a hyperbolic manifold $M$ by 
enumerating the combinatorial isomorphisms of the canonical retriangulation 
of $M$. Similarly, \texttt{SnapPy} uses it to check whether two manifolds 
are isometric.


\begin{definition}
The {\em isometry signature} of $M$ is the isomorphism signature of the 
canonical retriangulation if the canonical cell decomposition has 
non-simplicial cells. Otherwise, it is the isomorphism signature of the 
canonical cell decomposition itself.
\end{definition}

\begin{example}
The triangulation of \texttt{m004} given in the \texttt{SnapPy} census 
already is the canonical cell decomposition. Thus, the isometry signature 
of the manifold \texttt{m004} is the isomorphism signature of the census 
triangulation,
namely \texttt{cPcbbbiht} presented in Table~\ref{table:isomorph}. In the 
census of tetrahedral hyperbolic manifolds \texttt{m004} named 
$\otet02_{0001}$. Recall that this manifold is the figure-eight knot 
complement.\\
The cell decomposition for  \texttt{m202} given in the \texttt{SnapPy} 
census is not canonical. The isomorphism signature of its \texttt{SnapPy} 
triangulation is \texttt{eLMkbbdddemdxi} presented in 
Table~\ref{table:isomorph}. In the census of tetrahedral hyperbolic 
manifolds  \texttt{m202} named $\otet04_{0000}$. Observe, that 
$\otet04_{0000}$ is the complement of a 2-component link presented in 
Figure~\ref{f.tetlinks}. The isometry signature of \texttt{m202} is 
\texttt{jLLzzQQccdffihhiiqffofafoaa} that is realized by a triangulation 
with ten tetrahedra. 
\end{example}

\subsection{Computation of the tilt}

Consider an ideal triangulation ${\mathcal T} = \cup_{i} T_{i}$ of a cusped 
manifold $M$ with a shape
assignment for each tetrahedron, i.e., a $z_i\in\BC\setminus\{0,1\}$ 
determining an embedding of the tetrahedron $T_{i}$ as ideal tetrahedron in 
$\BH^3$ up to isometry. If the shapes fulfill the consistency equations 
(also known as gluing equations) in 
logarithmic form and have positive imaginary parts, we call the triangulation
together with the shape assignment a {\em geometric ideal triangulation}.
Thurston shows that a geometric ideal triangulation glues up to a complete
hyperbolic structure on $M$. Given a geometric ideal triangulation and a 
face $F$ of it, the 
tilt $\mathrm{Tilt}(F)$ is a real number defined by Weeks~\cite{Weeks}
which determines whether a given triangulation is proto-canonical and
which faces are transparent.

We now describe how to compute $\mathrm{Tilt}(F)$ following 
the notation in \cite{DHL} and use it to determine the
canonical retriangulation.

\subsubsection{Computation of a cusp cross section}

The ideal tetrahedra intersect the boundary of a neighborhood of a cusp
in Euclidean triangles and we call the resulting assignment of lengths to 
edges a {\em cusp cross
section}. We first compute
a cusp cross section $C_c$ for some neighborhood of each cusp $c$ by picking an
edge $e_j$ for each cusp and assigning length $e_j=1$ to it. We recursively
assign lengths to the other edges by using that the ratio of two edge
lengths
is given by the respective $|z_i^*|$ where $z_i^*$ is one of the edge
parameters $z_i, z'_i=\frac{1}{1-z_i}, z''_i=1-\frac{1}{z_i}$:
\begin{equation*}
e_l =e_k \cdot |z_i^*|. \label{eqn:EdgeRatioCusp}
\end{equation*}

\subsubsection{Computation of  the cusp area}

We can compute the area of each Euclidean triangle $t$ as 
\begin{equation*}
A(t)=\frac{1}{2} e_k^{\myChanged{2}} \cdot \mathrm{Im}(z_i^*)
\end{equation*}
where $e_k$ and $z_i^*$ are as above. The cusp area $A(C_c)$ of the cusp cross 
section $C_c$ is simply the sum of the areas $A(t)$ over all its Euclidean 
triangles $t$.

\subsubsection{Normalization of the cusp area}

\myChanged{We need to scale} each cusp cross section \myChanged{to have the 
same target area $A$. 
The new edge lengths and areas are given by
\begin{equation*}
e'_l = e_l \cdot\sqrt{\frac{A}{A(C_c)}}\quad\mbox{and}
\quad A'(t)=A(t)\frac{A}{A(C_c)}.
\end{equation*}}

\subsubsection{Computation of the circumradius for each Euclidean triangle}

Let $R^i_v$ denote the circumradius of the Euclidean triangle $t$ that is
the cross section of the tetrahedron $i$ near vertex $v\in\{0,1,2,3\}$. If 
$e'_j$, $e'_k$, and
$e'_l$ are the edge lengths of $t$, elementary trigonometry implies 
\begin{equation*}
R^i_v = \frac{e'_j e'_k e'_l}{4 A'(t)}.
\end{equation*}

\subsubsection{Computation of the tilt of a vertex}

Compute 
\begin{equation} \label{eqn:Tilt}
\mathrm{Tilt}(i, v) = R^i_v - \sum_{u\not=v} R^i_u\ 
\frac{\mathrm{Re}(z_i^*)}{|z_i^*|}
\end{equation}
where $z_i^*$ is the edge parameter for the edge from $u$ to $v$.

\subsubsection{Computation of the tilt of a face}

If the face $F$ opposite to vertex $v$ of tetrahedron $i$ is glued to that 
opposite of
$v'$ of tetrahedron $i'$, the tilt of the face is defined as 
\begin{equation*} 
\mathrm{Tilt}(F) = \mathrm{Tilt}(i,v) + \mathrm{Tilt}(i',v').
\end{equation*}

\subsubsection{Determination of transparent faces and canonical 
retriangulation}

Weeks proves that \cite{Weeks}
a geometric ideal triangulation is a geometric proto-canonical 
triangulation if all
$\mathrm{Tilt}(F)\leq 0$. In that case, a face $F$ is transparent if and 
only if
$\mathrm{Tilt}(F)=0$. 

\texttt{SnapPy} implements an algorithm to compute the canonical 
retriangulation.
It can be refactored so that it takes as input the opacities of the faces 
and is
purely combinatorial. In case of a geometric (!) proto-canonical 
triangulation,
Weeks' arguments in the \texttt{SnapPy} code prove that this
algorithm works correctly.

For all manifolds we encountered, several randomization trials were always
sufficient to ensure that the ideal triangulation returned by \texttt{SnapPy}'s
\texttt{canonize} is always geometric proto-canonical. Thus, the result of the
purely combinatorial canonical retriangulation algorithm is known to be correct
as long as we certify the input to be a geometric proto-canonical triangulation
with certified opacities of its faces.

\begin{remark}
Even though we can certify the results for all listed manifolds in the
tetrahedral census, it is not known if 
\begin{itemize}
\item
every cusped hyperbolic manifold has a geometric proto-canonical triangulation,
\item
every cusped hyperbolic manifold has a geometric ideal triangulation.
\end{itemize}
Moreover, it is known that \texttt{SnapPy}'s
implementation can give the wrong canonical retriangulation if we use as
input a non-geometric (!)
proto-canonical triangulation. As pointed out by Burton, the triangulation 
\texttt{x101} in the non-orientable cusped \texttt{SnapPy} census is such 
an example
where flat tetrahedra cause \texttt{SnapPy} to give an incorrect canonical 
retriangulation.\\
It is unclear to the authors which of the following factors contribute to 
the incorrect result:
\begin{itemize}
\item Numerical precision issues.
\item \texttt{SnapPy}'s extension of the above definition of 
$\mathrm{Tilt}(F)$ to flat tetrahedra (where some
$A(t)=0$ and thus $R^i_v=\infty$) using \texttt{CIRCUMRADIUS\_EPSILON}.
\item Week's arguments for the purely combinatorial part of the canonical
retriangulation algorithm seem to implicitly assume that there are no 
flat-tetrahedra.
\end{itemize}
\myChanged{
The existence of geometric triangulations of a hyperbolic manifold can be 
proven when some tetrahedra are allowed to be flat
 \cite{PetronioWeeks:partiallyFlatTrig}. It can also be proven virtually 
\cite{LuoSchleimerTillmann:virtualGeometricTrig}.}
\end{remark}

\begin{figure}
\begin{center}
\includegraphics[height=3.8cm]{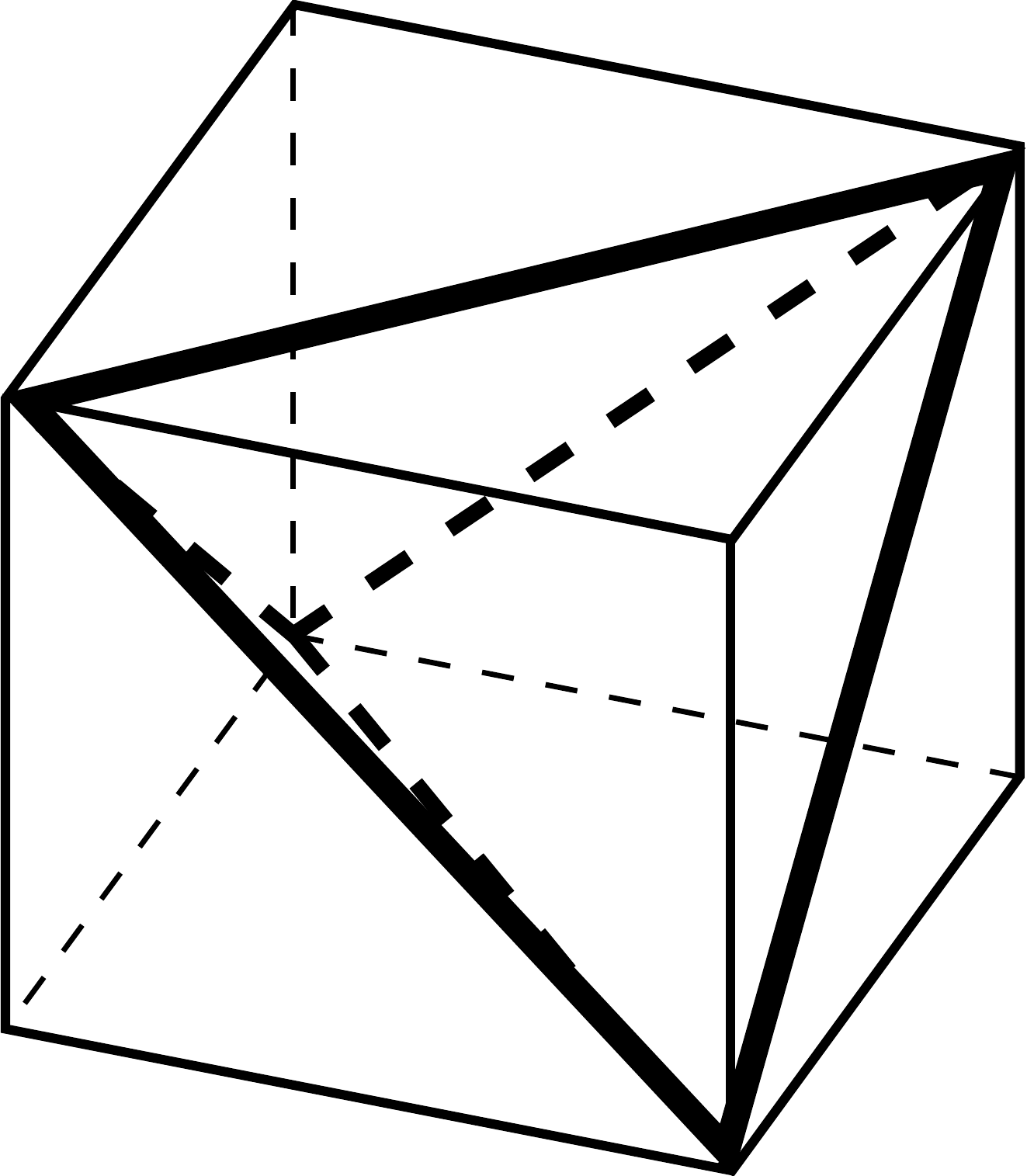}
\end{center}
\caption{Subdivision of a cube into 5 tetrahedra. For a regular ideal 
hyperbolic cube, all tetrahedra are again regular ideal. The subdivision 
introduced additional diagonals on the faces.\label{fig:tetInCube}}
\end{figure}

\begin{remark} \label{remark:burtonPair}
Call a manifold that can be decomposed into regular ideal cubes 
{\em cubical}. Recall that a regular ideal cube can be subdivided into 5 
regular ideal tetrahedra, see Figure~\ref{fig:tetInCube}. However, this 
does not imply that a cubical manifold is tetrahedral.\\
A counter-example is the manifold appearing in the census as 
$\texttt{x101}$ and $\texttt{x103}$. Its canonical cell decomposition 
consists of one regular ideal cube. As Burton explained~\cite{burton:x101}, 
$\texttt{x101}$ subdivides the cube into 5 regular ideal tetrahedra but 
needs to insert a flat tetrahedron to match the diagonals on the cube. 
Thus, it is not a tetrahedral manifold (but still has a tetrahedral 
double-cover $\ntet{10}_{0093}).$\\
\texttt{x103} splits the same cube into 6 non-regular tetrahedra and is 
a geometric proto-canonical triangulation.
\end{remark}

\subsection{Certification for tetrahedral manifolds}

Let $\sqrt{\BQ^+}$ denote the multiplicative group of all square roots of 
positive rational numbers and let $\BQ(\sqrt{\BQ^+})\subset \BC$ be the 
field generated by $\sqrt{\BQ^+}$.

\begin{lemma}
If we pick as target area $A=\sqrt{3}$, we have
for a geometric proto-canonical triangulation of a tetrahedral manifold $M$:
$$
z_i^* \in \BQ(\sqrt{-3});\qquad A(C_c)\in \BQ^+\sqrt{3}; 
$$  
$$ 
|z_i^*|, e_l, A(t), e'_l, A'(t), R^i_v \in \sqrt{\BQ^+};\qquad   
\mathrm{Tilt}(F)\in \BQ(\sqrt{\BQ^+}).
$$
\end{lemma}

\begin{proof}
$M$ and thus its universal cover can be decomposed into regular ideal 
tetrahedra. The resulting regular tessellation in $\BH^3$ can be chosen 
to have vertices at $\BQ(\sqrt{-3})$ (also see Section~\ref{sec.properties}), 
thus the shapes of any ideal triangulation of $M$ are in $\BQ(\sqrt{-3})$.\\
Develop a cusp cross section constructed above in $\BC$ such that the 
vertices of the edge set to length 1 are at 0 and 1. Then all vertices 
have complex coordinates in $\BQ(\sqrt{-3})$ and a fundamental domain in 
$\BC$ for the cusp is a parallelogram spanned by two complex numbers in 
$\BQ(\sqrt{-3})$. 
The area $A(C_c)$ of such a parallelogram is in $\BQ^+\sqrt{3}$.\\ 
The rest follows from the above formulas. 
\end{proof}

We can represent a $z_i^*$ exactly by $r_1+r_2\sqrt{\myChanged{-}3}$ with 
$r_1, r_2\in\BQ$. We can represent the other quantities exactly using 
Corollary \ref{cor.Galois}
below.

\begin{lemma}
\label{lem.Galois}
Let $p_1,\dots,p_r$ denote a list of distinct prime numbers and
$K=\BQ(\sqrt{p_1},\dots,\sqrt{p_r})$ denote the corresponding number field.
Then,
\begin{itemize}
\item[(a)]
$K/\BQ$ is Galois with Galois group $G(K/\BQ)=(\BZ/2\BZ)^r$.
\item[(b)]
If $\BQ \subset L \subset K$ is a subfield of $K$ such that $[L:\BQ]=2$
then $L=\BQ(\sqrt{p_I})$ where $p_I=\prod_{i \in I} p_i$ for some nonempty
$I \subset \{1,\dots,r\}$.
\item[(c)]
The $\BQ$-linear map
$$
\BQ(\sqrt{p_1}) \otimes_\BQ \BQ(\sqrt{p_2}) \otimes_\BQ \dots 
\otimes_\BQ \BQ(\sqrt{p_r}) \to K, \qquad x_1 \otimes \dots \otimes x_r \mapsto
\prod_{i=1}^r x_i
$$
is an isomorphism of $\BQ$-vector spaces.
\end{itemize} 
\end{lemma}

\begin{proof}
We will prove this by induction on $r$. When $r=1$ (a) is obvious and (b)
follows from the fundamental theorem of Galois theory~\cite[VI,Thm.1.2]{Lang}.

Assume that the lemma is true for $r-1$, and let 
$K_1=\BQ(\sqrt{p_1},\dots,\sqrt{p_{r-1}})$ and $K_2=\BQ(\sqrt{p_r})$.
Then, we claim that $K_1 \cap K_2 =\BQ$. Indeed, otherwise we have $K_2 \subset
K_1$ and by part (b) it follows that $\BQ(\sqrt{p_r})=\BQ(\sqrt{p_I})$
for some nonempty subset $I \subset \{1,\dots,r-1\}$. So, 
$\sqrt{p_r}=a + b \sqrt{p_I}$ for $a,b \in \BQ$. Squaring, we get 
$$
p_r = a^2 + b^2 p_I, \qquad a b = 0.
$$
If $a=0$ then $p_r=b^2 p_I$ and since $I$ is nonempty, it follows that
$p_i^2$ divides $p_r$ where $p_i, p_r$ are distinct primes, a contradiction.
If $b=0$ then $p_r=a^2$ and $p_r$ is a prime number, also a contradiction.
This shows that $K_1 \cap K_2 =\BQ$. Let 
$K=K_1 K_2 =\BQ(\sqrt{p_1},\dots,\sqrt{p_{r}})$ denote the composite field.
It follows by \cite[VI,Thm.1.14]{Lang} that $K$ is a Galois extension
with Galois group $G(K/\BQ)=G(K_1/\BQ) \times G(K_2/\BQ) = (\BZ/2\BZ)^r$.
This proves part (a) of the inductive part. Part (b) follows from part (a)
by the fundamental theorem of Galois theory \cite[VI,Thm.1.2]{Lang} and the
classification of all index 2 subgroups of $(\BZ/2\BZ)^r$.
Part (c) follows from part (a) and the induction hypothesis.
\end{proof}

Part (c) of Lemma \ref{lem.Galois} implies the following corollary.

\begin{corollary}
\label{cor.Galois}
Every element in $\BQ(\sqrt{\BQ^+})$ has a unique representative of the form 
\begin{equation}
\label{eq.ri}
r_1\sqrt{n_1}+\dots+r_k \sqrt{n_k}
\end{equation}
where $r_i\in\BQ\setminus 0$ and $n_1<\dots <n_k$ are square-free positive 
integers.
\end{corollary}

\begin{remark}
\label{rem.Galois.effective}
For the purpose of effective exact computation, we need an explicit way of
adding, subtracting, multiplying and dividing 
 expressions of the form ~\eqref{eq.ri}. 
This is obvious except for
division where we give the following algorithm: To compute $n/d$ where $n$ 
and $d$ are two such forms and $d$ contains a non-rational term 
$r_j\sqrt{n_j}$, pick a prime $p$ dividing $n_j$. We can write $d$ as 
$d_0 + \sqrt{p} d_1$ such that $d_0$ contains no term $r_i\sqrt{n_i}$ with 
$p | n_i$. We now have
$$
\frac{n}{d} = \frac{n (d_0-\sqrt{p} d_1)}{(d_0+\sqrt{p}d_1)(d_0-\sqrt{p} d_1)} 
= \frac{n (d_0-\sqrt{p} d_1)}{d_0^2 - p d_1^2}.
$$
The new denominator is simpler because it contains no more terms 
$r_i\sqrt{n_i}$ with $p | n_i$. Thus, by repeating this process we can 
eliminate all primes in the terms of the denominator.
\end{remark}

When we say {\em using interval arithmetics}, we mean:
\begin{enumerate}
\item We convert the exact representation of each quantity in 
$\BQ(\sqrt{\BQ^+})$, respectively, $\BQ(\sqrt{-3})$ to an interval $[a,b]$, 
respectively,  a complex interval $[a,b] + [a',b'] i$. These intervals have 
interval semantics: the true value of the quantity is guaranteed to be 
contained in the interval.
\item Any operations such as + or $\log$ are carried out such that interval 
semantics is preserved, i.e., the resulting interval is again guaranteed to 
contain the true value of the computed quantity.
\item An inequality involving an interval is considered certified only if 
it is true for all values in the interval. E.g., if the interval given for 
$x$ is $[a,b]$, then $x<0$ is certified only if $b<0$.
\end{enumerate}

We can now certify the geometric proto-canonical triangulation and the 
opacities of its faces. Our input is a candidate geometric proto-canonical 
triangulation obtained by calling \texttt{SnapPy}'s \texttt{canonize} on a 
tetrahedral manifold. We first guess exact values $z_i$ from the 
approximated shapes reported by \texttt{SnapPy}. Using those guesses, we verify
\begin{enumerate}
\item the rectangular form of the edge equations exactly,
\item $\mathrm{Im}(z_i)>0$ for each tetrahedron (using interval arithmetics),
\item $|e|<10^{-7}$ for each edge where $e$ is the error of the logarithmic 
form of the edge equation (using interval arithmetics),
\item all the equations~(\ref{eqn:EdgeRatioCusp}) exactly,
\item $\mathrm{Tilt}(F)<0$ (using interval arithmetics) for an opaque face, 
respectively, $\mathrm{Tilt}(F)=0$ (using exact arithmetics) for a 
transparent face.
\end{enumerate}

(1) implies that the error in (3) will be a multiple of $2\pi i$ so a 
small enough error implies that the logarithmic form of the edge equations 
is fulfilled exactly. Together with (2), this means that the tetrahedra 
yield a (not necessarily complete) hyperbolic structure. Completeness is 
ensured by (4) which checks that the cusp cross section is Euclidean. 
Checking (4) really means verifying that the recursion process to obtain 
the edge lengths could construct a consistent result. (5) certifies the 
geometric proto-canonical triangulation and the opacities of the faces.

\begin{remark}
Note that in the process, we actually produce complex intervals for the 
shapes from \texttt{SnapPy}'s approximations certified to contain the true 
values. We can do this because we know that the shapes are in the field 
$\BQ(\sqrt{-3})$ and thus can guess exact solutions and verify them exactly. 
An alternative method to obtain certified intervals from approximated shapes 
is the Krawczyk test implemented in \texttt{hikmot} \cite{hikmot}. We could 
not use it here though, because it cannot deal with non-orientable 
manifolds. The edge equations for a non-orientable manifold are polynomials 
in $z_i^*$ and ${1}/{\bar{z}_i^*}$.
\end{remark}

\begin{remark}
We could have also avoided guessing by tracking \texttt{SnapPy}'s 
algorithm to obtain a proto-canonical triangulation. We know that the 
shapes of the tetrahedral tessellation are all exactly represented by 
$\frac{1}{2}+\frac{1}{2}\sqrt{\myChanged{-}3} $ and that \texttt{SnapPy} 
is performing 2-3 and 3-2 moves during the algorithm. However, this would 
require changes to the \texttt{SnapPea} kernel since it does not report the 
sequence of moves it performed.
\end{remark}

For guessing a rational representation from an approximation, we use the 
\texttt{fractions} module shipped with \texttt{python}. It essentially 
computes the continued fraction for a given real number and evaluates it 
at a stage where the resulting denominator is less than a given bound 
(10000 in our case). For the (complex) interval arithmetics, we use 
\texttt{sage}. Our implementation in \texttt{python} is based on the script 
given in \cite{DHL}.

\subsection{Certification in the generic case}

Dunfield, Hoffman, L\myChanged{icat}a give an implementation in \cite{DHL} 
to certify a triangulation to be the canonical cell decomposition (which 
cannot contain non-tetrahedral cells). Though not needed here, we want to 
point out that their implementation can be both simplified and generalized 
to any number of cusps. 

They start with certified complex intervals for the shapes returned by 
\texttt{hikmot} \cite{hikmot}. But instead of following the complicated 
procedure in \cite[Section 3.7]{DHL}, one can simply apply interval 
arithmetics to the above equations to compute $\mathrm{Tilt}(F)$. The 
result is an interval $[a,b]$ for each $\mathrm{Tilt}(F)$ that is 
guaranteed to contain the true value of $\mathrm{Tilt}(F)$. If $b<0$ for 
each interval, then the $\mathrm{Tilt}(F)$ are certified to be less than 0, 
thus the given ideal triangulation is the canonical cell decomposition.

We provide a version of \texttt{canonical.py} here that implements this.

\section{Results of the implementation of algorithms}
\label{section4} 

We implemented the algorithms described in the previous section, see 
\cite{Goerner:tetcensus} for the resulting data. The longest algorithm 
to run was the enumeration of the \ctt s: the orientable case up to 25 
tetrahedra and the non-orientable one up to 21 tetrahedra each took about 
$\approx 6$ weeks CPU time and $\approx  70$Gb on a Xeon E5-2630, 2.3Ghz. 
The number of resulting \ctt s and \tman s are listed in 
Table~\ref{table:censusNumbers}. 

\subsection{Names of tetrahedral manifolds}

We give the \tman s names such as ``$\otet08_{0002}$'' (orientable), 
respectively, ``$\ntet02_{0000}$'' (non-orientable) with ``tet'' followed 
by the number of tetrahedra and an index. 
The different \ctt s corresponding to the same \tman{} are named with an 
additional index, e.g., ``$\otet08_{0002}\#0$'', ``$\otet08_{0002}\#1$''. We 
choose as canonical representative for an isometry class the first \ctt{}, 
e.g., $\otet08_{0002}\#0$ for the \tman{} $\otet08_{0002}$.\\
The indices are canonical: before indexing the \ctt s and \tman s, we 
first sort the \ctt s within an isometry class lexicographically by 
isomorphism signature and then sort the \tman s lexicographically by the 
isomorphism signature of their canonical representative.

\subsection{SnapPy census}

Our census of \tman s can be easily accessed from \texttt{SnapPy}. Simply 
change to the directory \texttt{snappy} accompanying this article and type 
\texttt{from tetrahedralCuspedCensus import *}. The two censuses 
\texttt{TetrahedralOrientableCusped\-Census} and 
\texttt{TetrahedralNonorientableCuspedCensus} have the same methods as any 
other census such as \texttt{OrientableCuspedCensus}. Here are examples of 
how to use them:

\begin{verbatim}
>>> from tetrahedralCuspedCensus import *
>>> M=TetrahedralOrientableCuspedCensus['otet02_0000'] # also m003
>>> TetrahedralOrientableCuspedCensus.identify(Manifold('m004'))
otet02_0001(0,0)
>>> len(TetrahedralOrientableCuspedCensus(tets=5)) # Number with 5 tets
2
>>> for M in TetrahedralOrientableCuspedCensus(tets=5):
...     print OrientableCuspedCensus.identify(M)
m410(0,0)
m412(0,0)(0,0)
>>> TetrahedralOrientableCuspedCensus.identify(Manifold("m208"))
>>>
\end{verbatim}

\myChanged{The last example shows that \texttt{m208} is not a tetrahedral 
manifold since it has only 5 tetrahedra and thus would be in the 
tetrahedral census. Note that \texttt{SnapPy}'s \texttt{is\_isometric\_to} 
is using numerical methods and can fail to find an isomorphism. To verify 
that \texttt{m208} is not tetrahedral, one can certify its isometry 
signature\footnote{\myChanged{
We plan a future publication describing how to generalize the techniques 
for certifying isometry signatures to all cusped hyperbolic manifold. 
The third named author has already incorporated this into SnapPy, beginning 
with version 2.3.2, see SnapPy documentation.}} and check that it is not 
in the data files \cite{Goerner:tetcensus} provided with this paper.}

\subsection{Regina files}

We also provide the census of \ctt s as two \texttt{Regina} files (for 
orientable and non-orientable) in the \texttt{Regina} directory 
accompanying this article. Each file groups the \ctt s first by number of 
tetrahedra and then by isometry class. The container for each isometry class 
contains the different \ctt s as well as the canonical retriangulation.

The \texttt{Regina} files can be inspected using the \texttt{Regina} GUI or 
the \texttt{Regina} python API. An example of how to traverse the tree 
structure in the file is given in \texttt{regina/example.py}.

\subsection{Morphisms} 
\label{section:morphisms}

Similarly to combinatorial isomorphism, we can define a 
{\em combinatorial homomorphism} between \ctt s, but without the 
requirement that different tetrahedra in the source go to the different 
tetrahedra in the destination. It assigns to each tetrahedron in the 
source a tetrahedron in the destination and a permutation in $S_4$ 
indicating which vertex of the source tetrahedron is mapped to which 
vertex of the destination tetrahedron. These permutations have to be 
compatible with the gluings of the source and destination tetrahedra. 
If the tessellations are connected and have no open faces, the source 
triangulation needs to have the same number of or a multiple of the number 
of tetrahedra as the destination. Topologically, a combinatorial 
homomorphism is a covering map that preserves the triangulation. We have 
implemented a procedure to list all combinatorial homomorphisms for a pair 
of triangulations in python.

We give a list of all pairs $(M, N)$ of \ctt s such that there is a 
combinatorial homomorphism from $M$ to $N$ as a text file 
\texttt{data/morphisms.txt}. We do not include the trivial pairs $(M,M)$ 
or pairs $(M,N)$ which factor through another \ctt{} as those can be 
recovered trivially through the reflexive and transitive closure. We also 
give some of the resulting graphs in \texttt{misc/graphs}. We discuss an 
example in more detail later in Section~\ref{section:cttCat}.

\section{Properties of \tman s}
\label{sec.properties}

\myChanged{
\subsection{Tetrahedral manifolds are arithmetic}
\label{sub.arithmetic}

Recall that two manifolds (or orbifolds) are commensurable if they have 
a common finite cover. Commensurability is an equivalence relation. 
The commensurability class of the figure-eight knot complement {\tt m004}
consists exactly of the cusped
hyperbolic orbifolds and manifolds with invariant trace field $\BQ(\sqrt{-3})$ that 
are arithmetic or, equivalently, that have integral traces
\cite[Theorem 8.2.3 and 8.3.2]{MR}. Thus, tetrahedral manifolds are
also arithmetic with the same invariant trace field since
\begin{lemma}
Tetrahedral manifolds are commensurable to {\tt m004}.
\end{lemma}

More precisely, the commensurability class of {\tt m004} also contains the orbifold 
$\mathfrak{R}=\BH^3/\mathrm{Isom}(\{3,3,6\})$ where the Coxeter group 
$\mathrm{Isom}(\{3,3,6\})$ is the symmetry group of the regular 
tessellation $\{3,3,6\}$ by regular ideal tetrahedra. This orbifold can be used
to characterize the tetrahedral manifolds in this commensurability class:

\begin{lemma}
A manifold $M$ is a covering space of $\mathfrak{R}$ if and only if it is 
tetrahedral.
\end{lemma}

\begin{proof}
A \ctt{} of a manifold $M$ lifts to the tessellation $\{3,3,6\}$ in its 
universal cover $\BH^3$. Thus, $\pi_1(M)$ is a subgroup of the symmetry group 
$\mathrm{Isom}(\{3,3,6\})$. Consequently, $M$ is a cover of $\mathfrak{R}$.

Conversely, a covering map $M\to\mathfrak{R}$ induces a \ctt{} on the manifold 
$M$ with the standard fundamental domain of $\mathfrak{R}$ lifting to the 
barycentric subdivision of the \ctt.
\end{proof}

\subsection{Implications of the Margulis Theorem}
\label{sec:Margulis}

Since {\tt m004} is arithmetic, Margulis Theorem implies that its 
commensurator is not discrete and thus the commensurability class of 
{\tt m004} contains no minimal element  
\cite{neumannReid:Topology90Arithmetic,MR,Walsh}. In particular, 
$\mathfrak{R}$ is not the minimal element of the commensurability class. 
We thus expect to see the following phenomena in the commensurability 
class containing the tetrahedral manifolds:
\begin{itemize}
\item Non-tetrahedral manifolds that are still commensurable with 
{\tt m004}. For example, the following manifolds in SnapPy's 
{\tt OrientableCuspedCensus} up to 8 simplices have this property:
$${\tt m208, s118, s119, s594, s595, s596, v2873, v2874}$$
\item Tetrahedral manifolds $M$ with different covering maps 
$M\to\mathfrak{R}$ inducing non-isomorphic \ctt s of the same manifold $M$.
\item Combinatorial tetrahedral tessellations ``hiding symmetries'', 
defined as follows.
\end{itemize}

\begin{definition}
A \ctt{} $T$ hides symmetries if the corresponding \tman{} $M$ has an 
isometry that is not induced from a combinatorial automorphism of  $T$. 
In other words, if there is an isometry $M\to M$ that does not commute 
with the covering map $M\to\mathfrak{R}$ corresponding to $T$. 
\label{def:hidesSymmetries}
\end{definition}

In this section, we will illustrate these phenomena using the tetrahedral 
census.

\begin{remark}
 By definition, the canonical cell decomposition and thus the canonical 
retriangulation sees all isometries, so we can detect this by checking 
that the number of combinatorial automorphisms of the canonical 
retriangulation is higher than those of the \ctt. To enable the reader to 
do this, the Regina file containing the tetrahedral census 
\cite{Goerner:tetcensus} includes the canonical retriangulation as well. 
The combinatorial automorphisms can be found using the method 
{\tt findAllIsomorphisms} of a Regina triangulation or {\tt find\_morphisms} 
in {\tt src/morphismMethods.py}.
 \end{remark}
}

\myChanged{

\begin{remark}
The minimum volume orientable cusped hyperbolic orbifold 
$\mathfrak{M}=\BH^3/\PGL(2,\BZ[\zeta])$ and the Bianchi orbifold 
$\mathfrak{B}=\BH^3/\PSL(2,\BZ[\zeta])$ of discriminant $D=-3$ where 
$\zeta=\frac{1+\sqrt{-3}}{2}$ are related to $\mathfrak{R}$ as follows 
with each map being a 2-fold covering \cite{neumannReid:SmallVolOrbs}:
$$
\mathfrak{B}\to\mathfrak{M}\to\mathfrak{R}.
$$
Similarly to $\mathfrak{R}$ being the quotient of $\BH^3$ by the symmetry 
group of the regular tessellation $\{3,3,6\}$, $\mathfrak{M}$ corresponds 
to orientation-preserving symmetries, and $\mathfrak{B}$ corresponds to the 
symmetry group of the regular tessellation $\{3,3,6\}$ after two-coloring 
the regular ideal tetrahedra.

Thus, the manifold covering spaces of $\mathfrak{M}$ correspond to the 
orientable \ctt s, and the manifold covering spaces of $\mathfrak{B}$ 
correspond to orientable \ctt s whose tetrahedra can be two-colored. Regina 
displays the dual 1-skeleton of a triangulation in its UI under 
``Skeleton: face pairing graph'', so we can check whether a \ctt{} is a 
cover of $\mathfrak{B}$ by testing whether the graph Regina shows is 
two-colorable.  For example, 
all orientable \ctt s with \myChanged{fewer} than 5 tetrahedra are covers of 
$\mathfrak{B}$. But $\otet05_{0000}$ and $\otet06_{0000}$ are not.
\end{remark}

\begin{remark} 
Related results include: \cite{BMR95} show that all once-punctured torus 
bundles in the commensurability class of the figure eight-knot complement 
{\tt m004} are actually cyclic covers of the tetrahedral manifolds 
{\tt m003} and {\tt m004} and thus tetrahedral. The non-arithmetic 
hyperbolic once-punctured torus bundles are studied in 
\cite{goodmanHeardHodgson:CommensuratorsOfCuspedHyperbolicManifolds} 
where an algorithm is given to compute the commensurator of a cusped 
non-arithmetic hyperbolic manifold. \cite{RV} study symmetries of or 
hidden by cyclic branched coverings of 2-bridge knots.
\end{remark}

}

\subsection{The category of \ctt s} 
\label{section:cttCat}

\myChanged{ To study the commensurability class containing the tetrahedral 
manifolds, we think of it as a category. For this,} recall
 the notion of a combinatorial homomorphism from 
Section~\ref{section:morphisms}. On the underlying topological space, a 
combinatorial homomorphism is a covering map. We thus get two categories 
with a forgetful functor \myChanged{$\mathcal{T}\to\mathcal{M}$}:

\begin{definition} 
\label{def:categories}
The {\em category \myChanged{$\mathcal{M}$} of \myChanged{manifolds 
commensurable with \tman s}} has as objects manifolds commensurable 
\myChanged{with} {\tt m004} and as morphisms covering maps.\\
The {\em category \myChanged{$\mathcal{T}$} of \ctt s} has as objects 
\ctt s and as morphisms combinatorial homomorphisms. 
\myRemoved{changed order of $\mathcal{T}$ 
and $\mathcal{M}$}
\end{definition}

\myRemoved{Removed Remark}
\myComment{
\begin{remark}
The first category is equivalent to the category whose objects are 
covering maps $M\to \mathfrak{R}$ (with $M$ begin a manifold) and 
whose morphisms are commutative diagrams of covering maps 
$$
\xymatrix{M \ar[rd] \ar[rr] & & M' \ar[ld]\\ & \mathfrak{R}}
$$
and the forgetful functor takes $M\to \mathfrak{R}$ to $M$.
\end{remark}}

We show a small part of these categories in Figure~\ref{fig:catCtt} 
and observe:
\begin{itemize}
\item $\otet04_{0001}\#0$ has two 2-covers (indicated by the solid arrows) 
giving two different triangulations $\otet08_{0002}\#0$ and 
$\otet08_{0002}\#1$. These triangulations are not combinatorially 
isomorphic but yield isometric manifolds (indicated by the dashed line).
\item The figure-eight knot complement, $\otet02_{0001}\#0$, and its 
sister, $\otet02_{0000}\#0$, have a common cover $\otet04_{0002}\#0$. More 
general, any two \ctt s have a common cover \ctt {} \myChanged{as they} 
are in the same commensurability class.
\item $\otet02_{0001}\#0$ and $\otet02_{0000}\#0$ show that the graph is a 
poset with more than one minimal element. In fact, most \ctt s in our 
census are minimal elements and we conjecture that there are infinitely 
many such minimal elements. 
\item The figure also shows a manifold \texttt{m208}, \myChanged{which is 
non-tetrahedral}. \myChanged{However, as with any manifold in this 
commensurability class, it still has a tetrahedral covering space, here} 
$\otet08_{0010}\#0$ (the arrow has to be dashed because \texttt{m208} is not 
tetrahedral so the map is not a combinatorial homomorphism).
\end{itemize}

\begin{remark} 
\label{remark:ExampleHiddenSym}
The last example shows that the \myChanged{\ctt{}} $\otet08_{0010}\#0$ 
\myChanged{hides symmetries as in Definition~\ref{def:hidesSymmetries}}. 
To see this, notice that the covering space $\otet08_{0010} \to 
\texttt{m208}$ is 2-fold, thus regular and \texttt{m208} is the quotient 
of $\otet08_{0010}$ by the group $G=\BZ/2\BZ$ of deck transformations. If 
$G$ preserved the \ctt{} $\otet08_{0010}\#0$, the quotient \texttt{m208} would
have an induced \ctt. But \texttt{m208} is not tetrahedral, thus the 
nontrivial
element of $G$ is a symmetry of $\otet08_{0010}\#0$ which is not a 
combinatorial homomorphism.
\end{remark}

\begin{figure}
\begin{center}
\scalebox{0.39}{\includegraphics{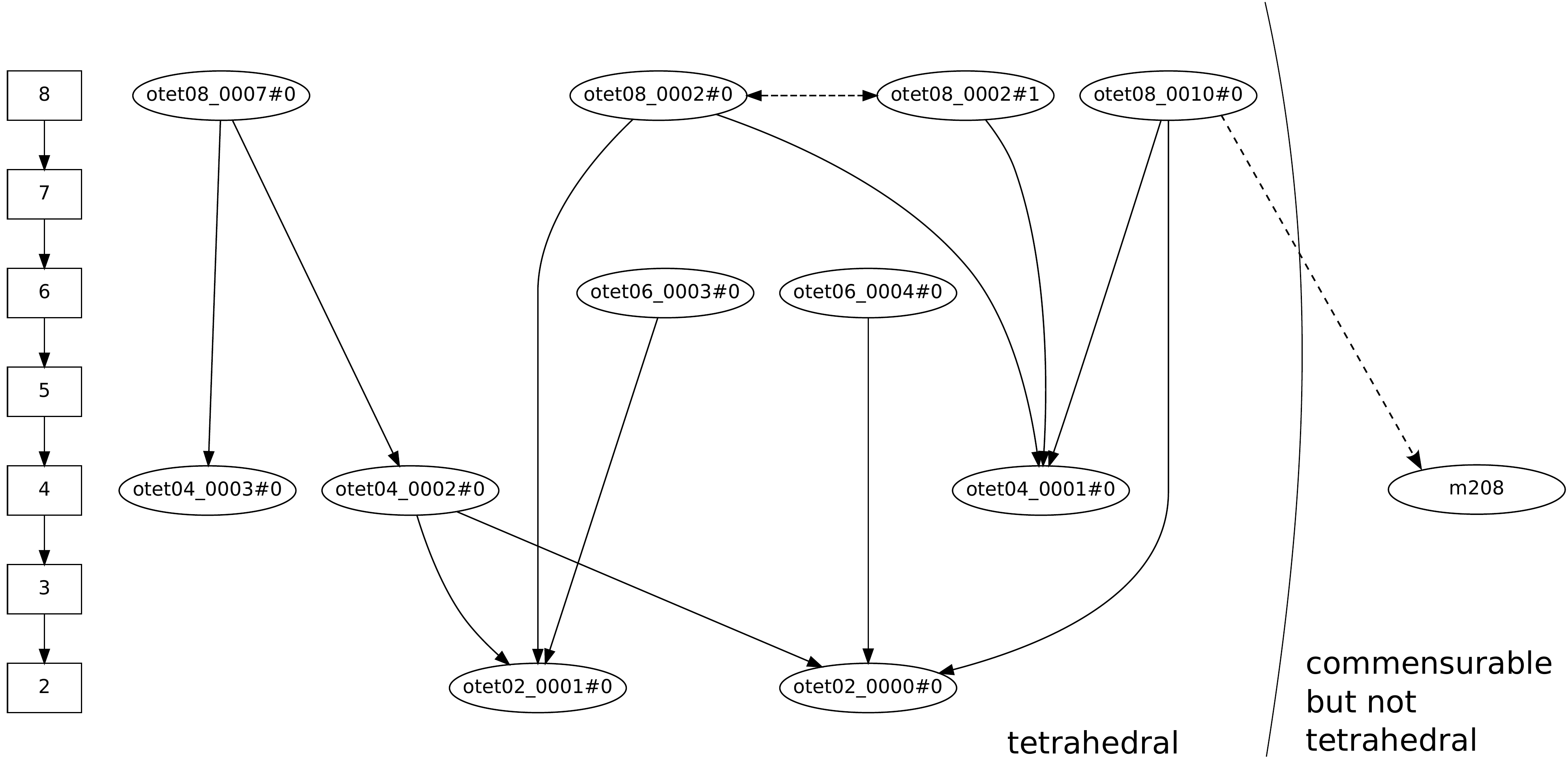}}
\end{center}
\caption{A small part of the category 
\myChanged{$\mathcal{T}$} of \ctt s (solid arrows) 
and the larger category \myChanged{$\mathcal{M}$} of \myChanged{manifolds 
commensurable with {\tt m004}} (dashed arrows). Multiple morphisms 
between two objects are collapsed to just one arrow, automorphisms and 
morphisms factoring through another object are dropped.
\label{fig:catCtt}}
\end{figure}

\myRemoved{Small improvements to Figure 2.}

\subsection{Canonical cell decompositions}

\subsubsection{Examples}

The canonical cell decomposition of a \tman{} can:
\begin{itemize}
\item Be a {\bf \ctt}.\\
{\bf Examples:} $\otet02_{0000}$ and $\otet10_{0010}$. The latter one has 
two \ctt s, $\otet10_{0010}\#0$ being the canonical cell decomposition.
\item Be a {\bf coarsening} of a \ctt{}.\\ (i.e., the \ctt{} is a 
subdivision of the canonical cell decomposition.)\\
{\bf Example:} $\otet05_{0001}$. The canonical cell decomposition consist 
of single regular ideal cube that can be subdivided into 5 tetrahedra 
(see Figure~\ref{fig:tetInCube}) such that the diagonals introduced on 
the faces are compatible. This yields the unique (up to combinatorial 
isomorphism) \ctt{} for this manifold. We elaborate on the relationships 
to cubes below.
\item {\bf Neither} of the above.\\
 In which case, the canonical cell decomposition can still
\begin{itemize}
\item Consists of (non-regular) tetrahedra.\\
{\bf Example:} $\otet08_{0010}$.
\item Contain cells which are not tetrahedra\\
{\bf Example:} $\otet08_{0001}$. Its canonical cell decomposition contains 
some hexahedra obtained by gluing two non-regular tetrahedra.
\end{itemize}
\end{itemize}

\subsubsection{Cubical manifolds}

\myChanged{
Recall from Remark~\ref{remark:burtonPair} that a manifold was called 
cubical if it can be decomposed into regular ideal cubes. 
Figure~\ref{fig:tetInCube} showed that there are two choices of picking 
alternating vertices of a cube, which span a tetrahedron and thus yield a 
subdivision of a regular ideal cube into 5 regular ideal tetrahedra. Even 
though each cube of a combinatorial cubical tessellation can be subdivided 
into regular ideal tetrahedra individually, this only yields a \ctt{} if 
the choices made are compatible with the face-pairings of the combinatorial 
cubical tessellation. We saw $\otet05_{0001}$ above as an example where this 
was possible and {\tt x103} in Remark~\ref{remark:burtonPair} as an example 
where this was impossible.

}

\myComment{
Even though each ideal cube can be divided into 5 regular ideal tetrahedra 
individually, the remark showed an example (\texttt{x103}) where this did 
not yield a \ctt{}. An example of the opposite where this yields a \ctt{} 
is the above manifold $\otet05_{0001}$.\\
}

If a manifold is both tetrahedral and cubical, the canonical cell 
decomposition can actually consist of regular cubes or regular ideal 
tetrahedra (or neither). This is illustrated by the two cubical links given 
by Aitchison and Rubinstein \cite{AR:dodecahedral}:
\begin{itemize}
\item The canonical cell decomposition of the complement $\otet10_{0011}$ 
of the alternating 4-chain link \texttt{L8a21} (see Figure~\ref{f.tetlinks}) 
consists of two regular ideal cubes. 
\item The complement $\otet10_{0006}$ of the other cubical link 
\texttt{L8a20} (see Figure~\ref{f.tetlinks}) admits two \ctt s up to 
combinatorial isomorphism, one of which is equal to the canonical cell 
decomposition.
\end{itemize}

\begin{remark}
\myChanged{Figure~\ref{fig:tetInCube} also shows that the choice of 5 
regular ideal tetrahedra to subdivide a cube hides symmetries of the cube, 
namely, the rotation by $\pi/2$ of the cube that takes one choice to the 
other. This rotation is an element in the commensurator but not in the 
normalizer of $\mathrm{Isom}(\{3,3,6\})$ and thus a hidden symmetry of 
$\mathfrak{R}$. A \ctt{} arising as subdivision of a combinatorial 
cubical tessellation can hide the symmetries of the combinatorial cubical 
tessellation corresponding to this rotation, i.e., there can be symmetries 
of the combinatorial cubical tessellation that are not symmetries of the \ctt.
}

\myComment{
Cubical manifolds can actually have symmetries visible in the cubical 
structure but not the tetrahedral one. To describe this hidden symmetry,
imagine a regular ideal cube and pick a set of alternating vertices. 
These alternating vertices span a regular ideal tetrahedron, see 
Figure~\ref{fig:tetInCube}. A rotation by $\pi/2$ takes the tetrahedron 
to the tetrahedron spanned by the complementary set of vertices. This 
rotation is a hidden symmetry of the tetrahedron, and thus $\mathfrak{R}$.\\
The \ctt{} of $\otet10_{0011}$ actually realizes this hidden symmetry. 
In other words, its two regular ideal cubes can be divided each into 5 
regular ideal tetrahedra in a way such that the diagonals introduced on 
the faces of each cube match. If we flip all the diagonals on all the 
faces, we obtain a different subdivision. Thus, when subdividing to obtain 
the \ctt, a choice was made and the manifold $\otet10_{0011}$ has a symmetry 
taking one of these choices to the other. This symmetry is hidden in the 
\ctt{}.\\
}
\myChanged{An example of this is $\otet10_{0011}$. }
Other examples are obtained by subdividing the cubical regular 
tessellation link complements $\mathcal{U}_{1+\zeta}^{\{4,3,6\}}$, 
$\mathcal{U}_2^{\{4,3,6\}}$\myChanged{,} and $\mathcal{U}_{2+\zeta}^{\{4,3,6\}}$ 
classified in \cite{Goerner}. By definition, each of these three manifolds 
can be decomposed into ideal regular cubes such that each flag of a cube, 
an adjacent face and an edge adjacent to the face can be taken to any other 
flag by a symmetry. In particular, these manifolds contain a symmetry 
flipping the diagonals of the faces of the cubes.
\end{remark}

\subsubsection{Canonical \ctt s}

\myChanged{We call a \ctt{} a regular tessellation if it corresponds} to a regular 
covering space of $\mathfrak{R}$ or $\mathfrak{M}$. This is equivalent to 
saying that the combinatorial automorphisms act transitively on flags 
consisting of a tetrahedron, an adjacent face and an adjacent edge (we drop 
the vertex in the flag to allow chiral \ctt s) \cite{Goerner}.

\begin{lemma}
Consider a \ctt{} \myChanged{$T$}. \myChanged{$T$ is equal to the 
canonical cell decomposition of the corresponding \tman{} $M$ if $T$ is a 
regular tessellation
or if $M$ has only one cusp. In particular, a \tman{} 
with only one cusp has a unique \ctt. If $T$ is equal to the canonical cell 
decomposition, then $T$ hides no symmetries.}
\myComment{
\begin{center}
\includegraphics{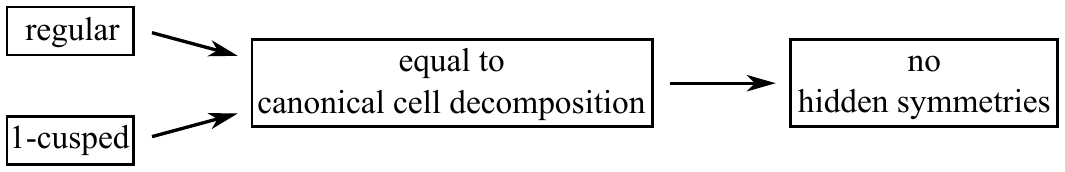}
\end{center}}
\end{lemma}

\begin{proof}
\myChanged{
Recall from Section~\ref{sec:defCanonicalCell} that the canonical cell 
decomposition relies on choosing cusp neighborhoods of the same volume 
for each cusp. If $T$ is regular, then each cusp neighborhood intersects 
$T$ in the same triangulation. This is also true if $M$ has only one cusp 
and there is only one cusp neighborhood to choose. Thus, each end of a 
tetrahedron intersects the cusp neighborhoods in the same volume. $T$ lifts 
to the regular tessellation $\{3,3,6\}$ of $\BH^3$ and the cusp 
neighborhoods lift to horoballs with the same symmetry. Hence, the canonical 
cell decomposition is equal to $T$. The other statement follows from the 
canonical cell decomposition not hiding any symmetries by definition.}
\myComment{
When computing the tilts in Equation~\ref{eqn:Tilt} for a 1-cusped \ctt{}, 
there is no rescaling of cusp cross sections necessary and all tetrahedra 
are regular. Thus, all $R^i_v$ are equal and all 
$\mathrm{Im}(z_i^*)=\sqrt{3}/2$ making all tilts negative. Hence, 1-cusped 
tetrahedral manifold has a unique \ctt{} which is always equal to the 
canonical cell decomposition.

Since all symmetries of a manifold are visible in the canonical cell 
decomposition, all symmetries are combinatorial isomorphisms if the \ctt{} 
is equal to the canonical cell decomposition. Thus, such a \ctt{} realizes 
no hidden symmetries.}
\end{proof}

\myChanged{
\begin{remark}
For some cubical tessellations such as $\mathcal{U}_{1+\zeta}^{\{4,3,6\}}$, 
$\mathcal{U}_2^{\{4,3,6\}}$, and $\mathcal{U}_{2+\zeta}^{\{4,3,6\}}$, we can 
partition the cusps into two disjoint sets such that no edge connects two 
cusps of the same set. If, in the construction of the canonical cell 
decomposition, we now pick for cusps in one set cusp neighborhoods of a 
volume slightly different from those for cusps in the other set, we no 
longer obtain the cubical tessellation but one of the two subdivided \ctt s 
depending on which set of cusps we favored.
\end{remark}
}

\section{Tetrahedral links}
\label{sec.links}

\subsection{Some facts about tetrahedral links} 

Consider a cusped 3-manifold $M$, i.e., the interior of a compact 
3-manifold $\bar{M}$ with boundary $\partial\bar{M}$ a disjoint union of 
tori. We say that $M$ is a \emph{homology link complement} if the long 
exact sequence in homology associated to $(\bar{M},\partial \bar{M})$ is 
isomorphic to that of the complement of a link in $S^3$. Let 
$i:\partial\bar{M}\to \bar{M}$ denote the inclusion of the boundary. We 
thank C. Gord\myChanged{o}n for pointing out to us that (b) implies (d). 


\begin{lemma} 
\label{lem.homology.link}
Let $M$ be a cusped 3-manifold. The following are equivalent:
\begin{enumerate}
\renewcommand{\theenumi}{\alph{enumi}}
\item $M$ is a homology link complement.
\item $H_1(M;\BZ)=\BZ^c$ where $c$ is the number of cusps.
\item The cuspidal homology 
$H^{\mathrm{cusp}}_1(M)=H_1(\bar{M};\BZ)/\mathrm{Im}(i_*)$ vanishes.
\item $M$ is the complement of a link in an integral homology sphere.
\end{enumerate}
\end{lemma}

\begin{proof}
(a) implies (b) since $H_1(\partial \bar{M})\cong\BZ^{2c}$ determines $c$ 
and $H_1(M)=\BZ^c$ for a link complement in $S^3$. The equivalence of (b) 
and (c) was shown in \cite[Lem.6.9]{Goerner}. To prove that (b) implies (d), 
we work
by induction on $c$. For $c=0$, $M$ is a homology sphere and thus the 
complement of the empty link.
Assuming it is true for $c-1$, pick a component $T$ of $\partial \bar{M}$ 
and let $H$ be
the image of $H_1(T;\BZ)$ in $H_1(\bar{M};\BZ)$ under the map induced by 
inclusion. By Poincare duality, 
$H$ has rank $1$ or $2$ (apply 
\cite[Chapter VI, Theorem 10. 4]{bredon:top_and_geo} to $\bar{M}$ with 
all boundary components but $T$ Dehn-filled). Now we claim that $H$ 
contains a rank $1$ direct summand of $H_1(\bar{M};\BZ)$ (so one can now do a 
Dehn filling on $T$ to reduce $c$ by 1). For if not, then $H$ is contained in 
$p H_1(\bar{M};\BZ)$ for some prime $p$. Then
$H_1(T;\BZ_p)$ maps trivially in $H_1(\bar{M};\BZ_p)$, contradicting duality.

It is left to show that (d) implies (a). This follows easily from 
Alexander duality \cite{BZ}.
\end{proof}

A homology link $M$ is the complement of a link in the 3-sphere if 
and only if there is a Dehn-filling of it with trivial fundamental group. 
In that 
case, the filling is a homotopy 3-sphere, hence a standard 3-sphere (by 
Perelman's Theorem), and the 
link is the complement of the core of the filling. \texttt{SnapPy} can 
compute the homology of a hyperbolic manifold as well as a presentation of 
its fundamental group, before or after filling. Note that links are in 
general not determined by their complement, i.e., there are 3-manifolds 
that arise as the complement of infinitely many different links~\cite{Gordon}. 
On the other hand, 
the only tetrahedral knot is the figure-eight knot. This follows from the 
fact that 
tetrahedral manifolds are arithmetic, and the only arithmetic knot is the 
figure-eight  
knot~\myChanged{\cite[Theorem 2]{reid:arithmeticity}}. \myRemoved{Changed 
reference}

\subsection{A list of tetrahedral links}

Of the 124 orientable tetrahedral manifolds with at most 12 tetrahedra, 
27 are homology links and \texttt{SnapPy} identified 13 of them with link 
exteriors in its census. Of the remaining 14 homology links, 
\begin{itemize}
\item 
$\otet04_{0000}$ is the Berge \myChanged{manifold, the complement of a link 
in \cite{MP:magicmfdfilling}}\myRemoved{Changed reference}, 
\item
11 are link complements, with corresponding links shown in 
Figure~\ref{f.tetlinks} and~\ref{f.tetlinks12}.\\ 
(These links were found by drilling some curves until the manifold could
be identified as a complement of a link in SnapPy's \texttt{HTLinkExteriors}. 
We then found a framing of some components of the link
such that Dehn-filling gives back the \tman. This gives us a Kirby diagram 
of the \tman. Using the Kirby Calculator \cite{kirbycalculator}, we 
successfully removed all Dehn-surgeries and obtained a link.)
\item $\otet08_{0003}$ and $\otet10_{0023}$ 
(with 2 and 1 cusps respectively) are not link complements.\\
(This can be shown
using {\tt fef\_gen.py} \myChanged{based on} \cite{MPR:fivechainfilling} 
and available from \cite{fefgen}\myRemoved{Changed reference} to list all 
exceptional slopes.\myRemoved{No footnote as this bug has been resolved in 
version 1.3 of {\texttt fef\_gen}}
and then compute homologies for those.)
\end{itemize}

The data in Table~\ref{table:censusNumbers} also suggest: 

\begin{conjecture}
Every tetrahedral link complement has an even number of tetrahedra (i.e., 
a corresponding \ctt{} has an even number of tetrahedra).
\end{conjecture}

\subsection{A remarkable tetrahedral link}
\label{sub.remarkable}

Of the 11580 orientable \tman s with at most 25 tetrahedra, 885
are homology links, and have at most 7 cusps. There is a unique \tman \
with 7 cusps, $\otet20_{0570}$, which is a link complement, and a 
2-fold cover of the minimally twisted 5-chain link $L10n113=\otet10_{0027}$.
This remarkable link is shown in Figure~\ref{fig.remarkable}.

\begin{figure}[!hptb] 
\centering{
\includegraphics[height=0.12\textheight]{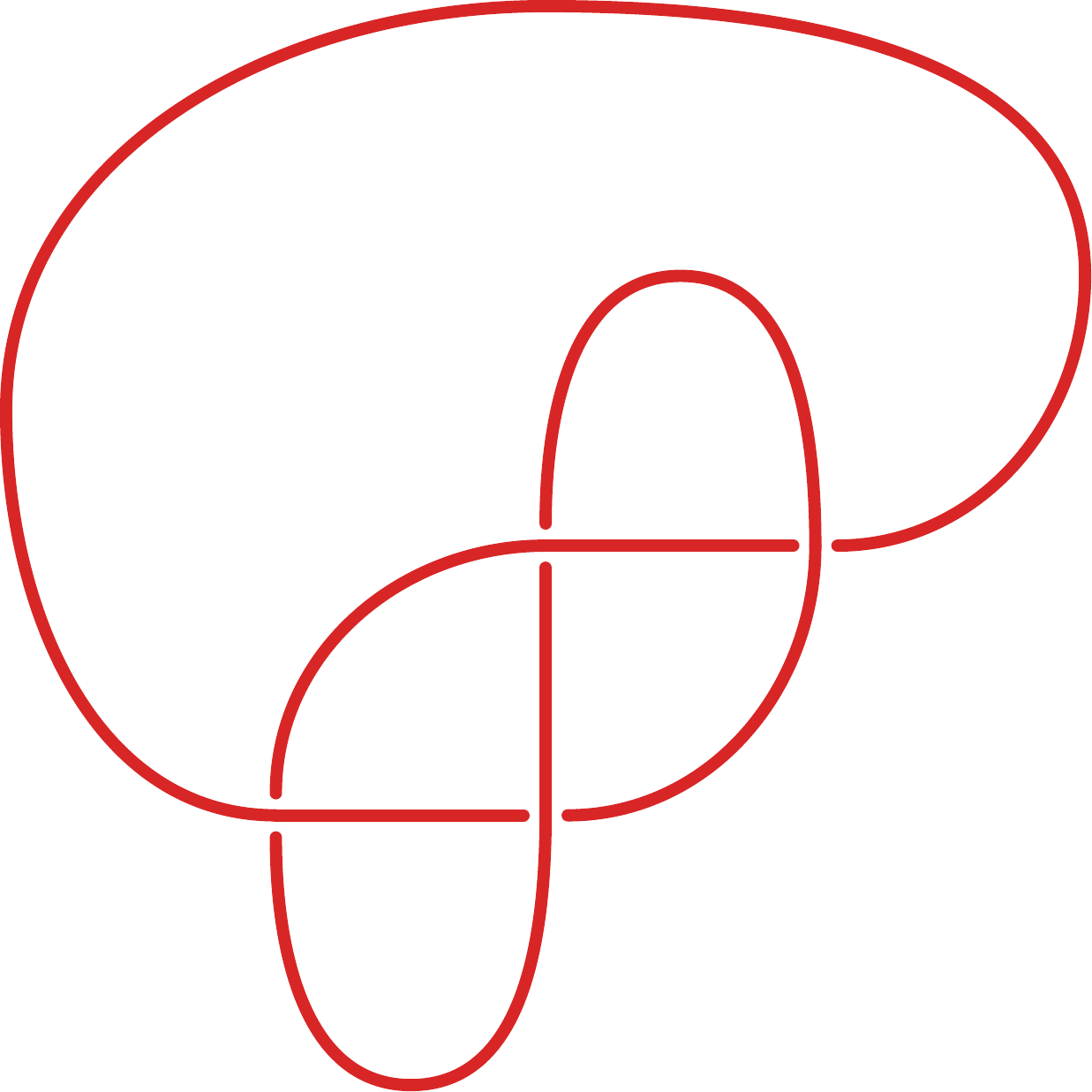}\hspace{5mm} 
\includegraphics[height=0.12\textheight]{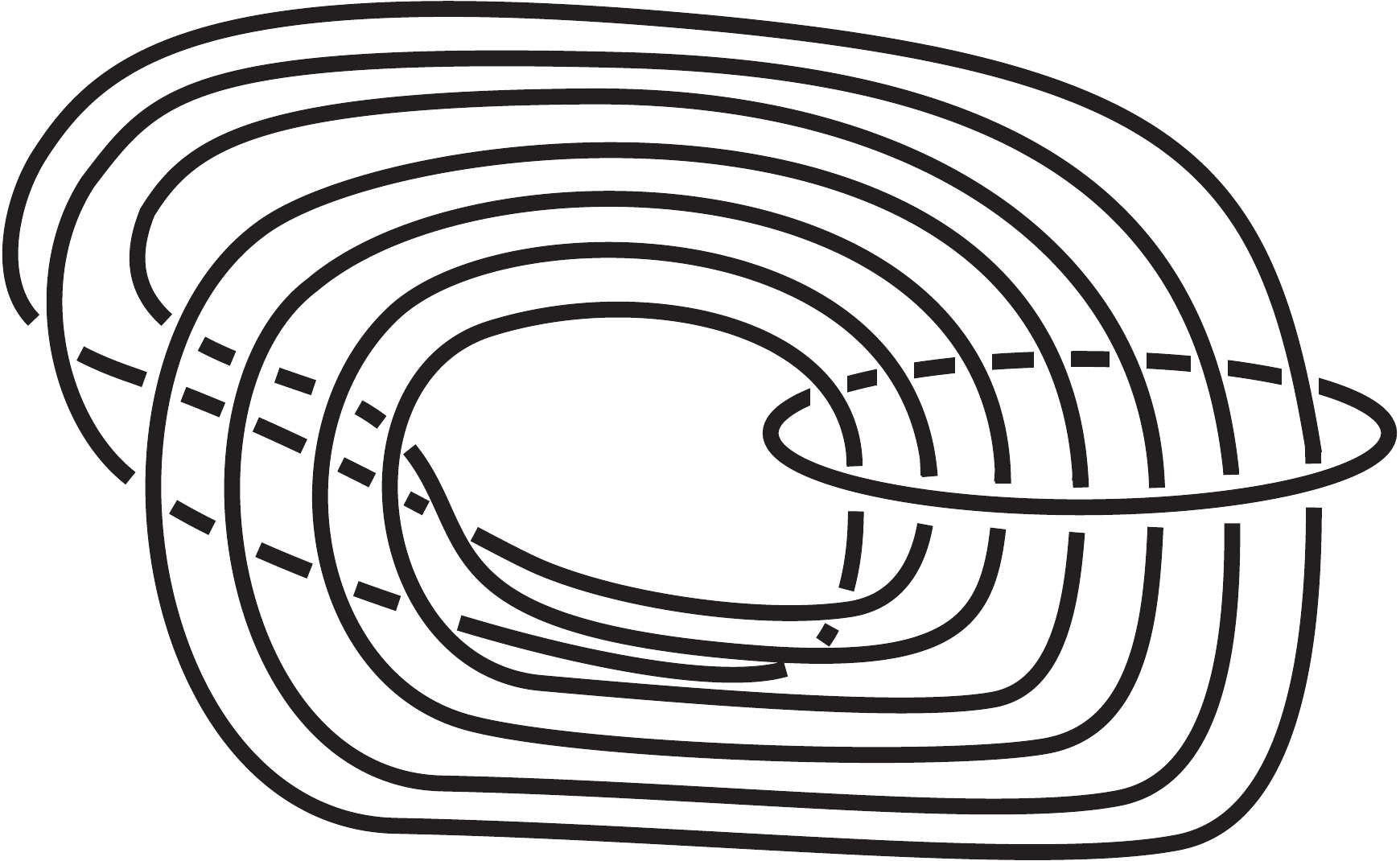}\hspace{5mm}
\includegraphics[height=0.12\textheight]{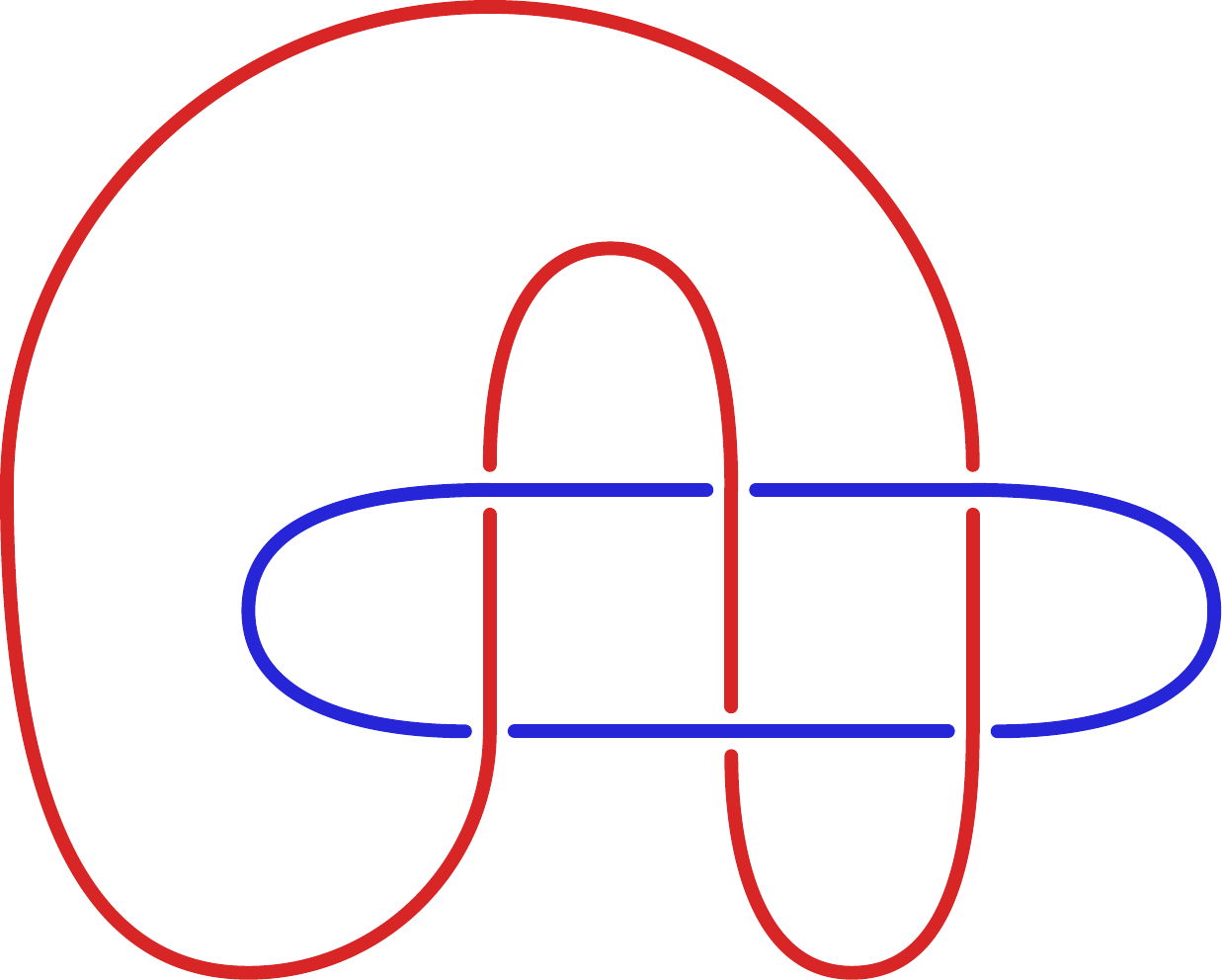}\hspace{5mm}
\includegraphics[height=0.12\textheight]{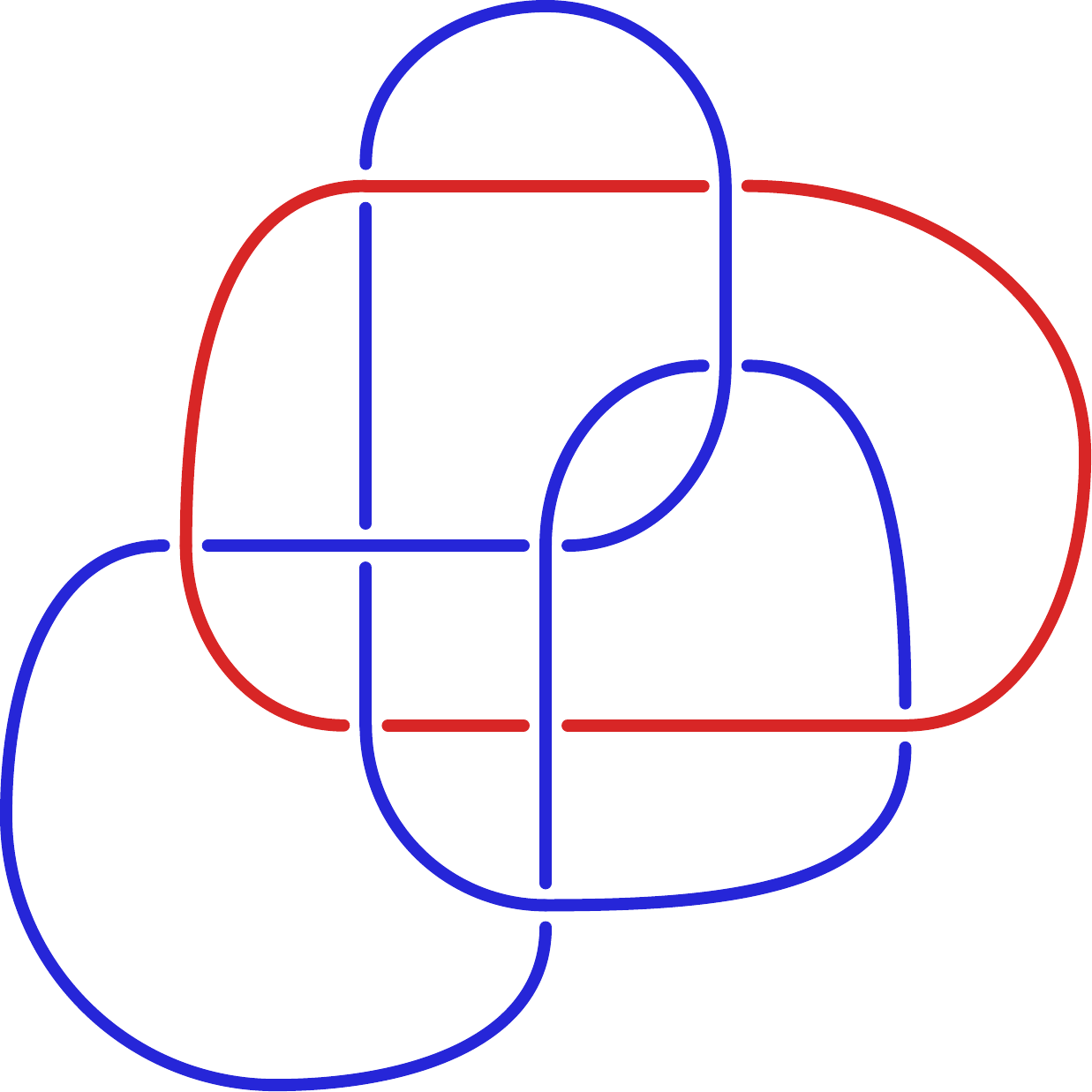}
\\
$$
\otet02_{0001}(K4a1) \qquad \quad
\otet04_{0000} \qquad \qquad \quad
\otet04_{0001}(L6a2) \quad
\otet08_{0002}(L10n46)
$$ 
\\
\includegraphics[height=0.14\textheight]{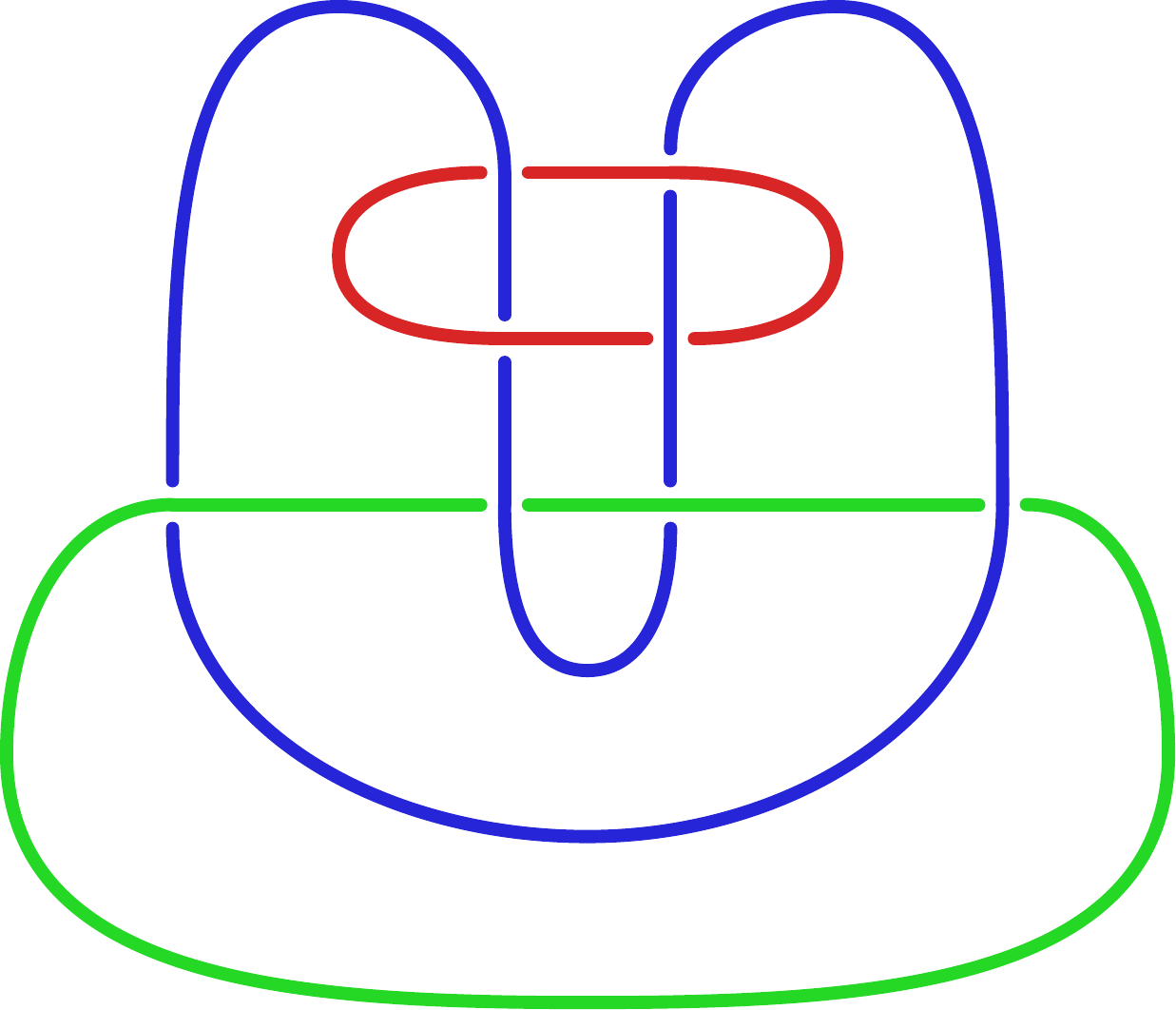}\hspace{2mm} 
\includegraphics[height=0.14\textheight]{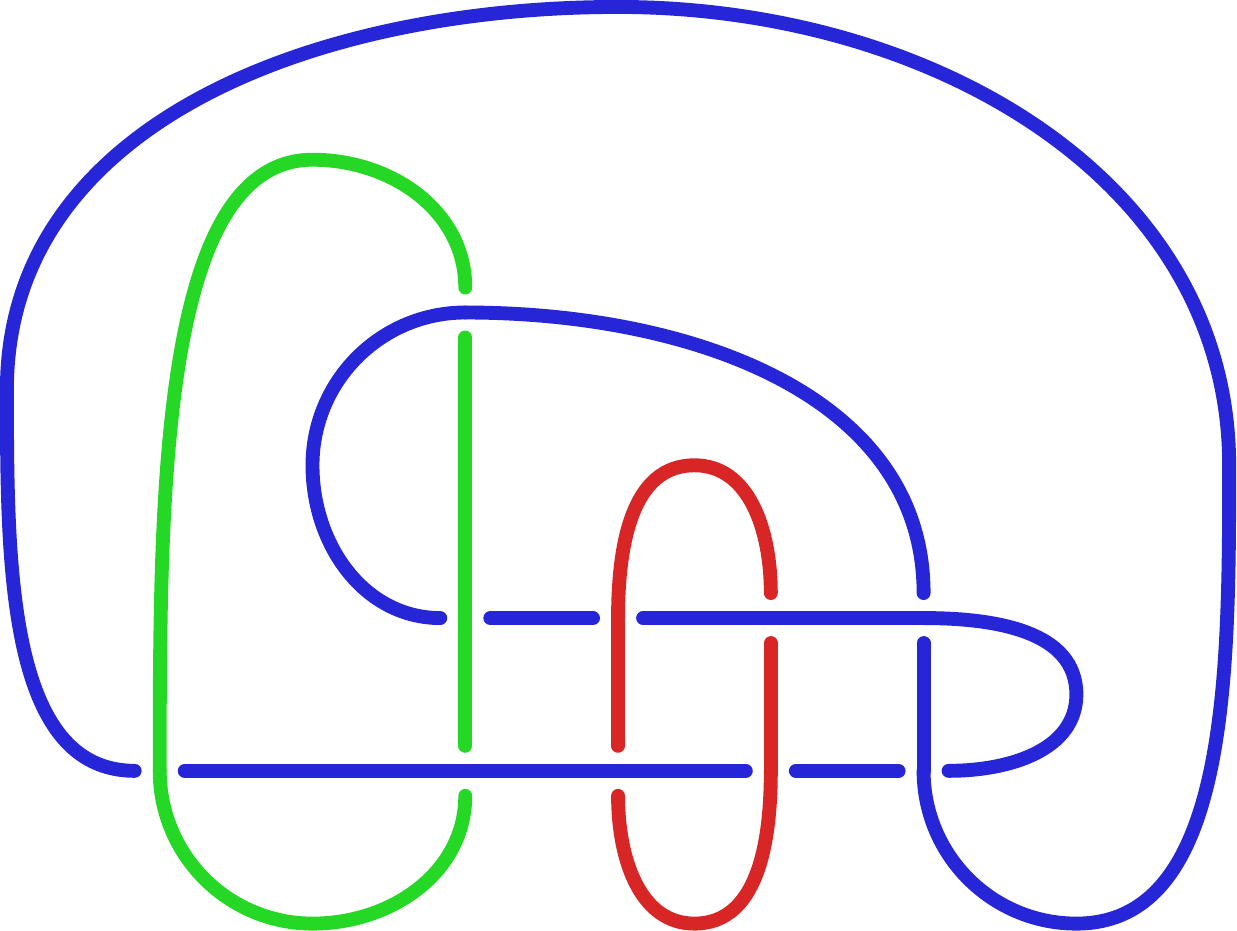}\hspace{2mm}
\includegraphics[height=0.14\textheight]{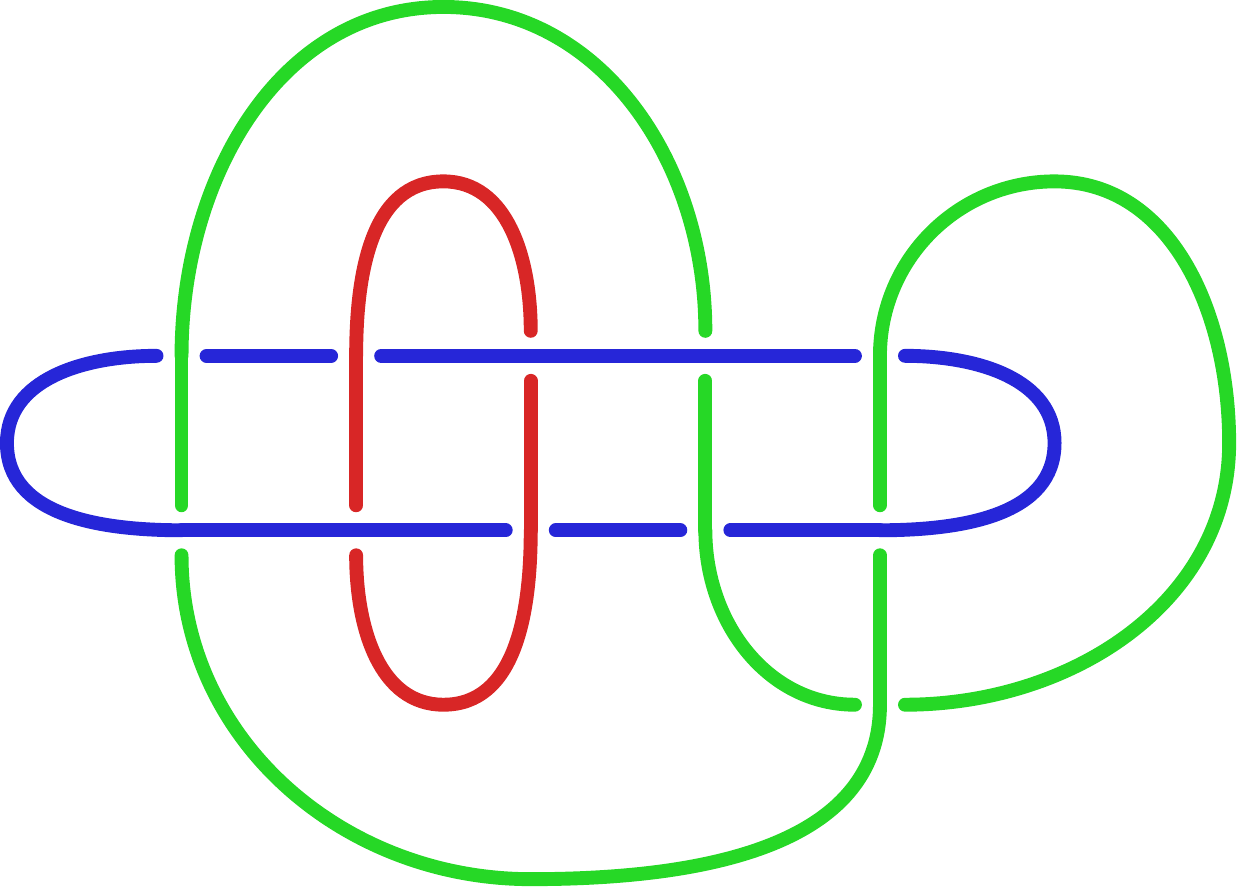}\hspace{2mm}
\includegraphics[height=0.14\textheight]{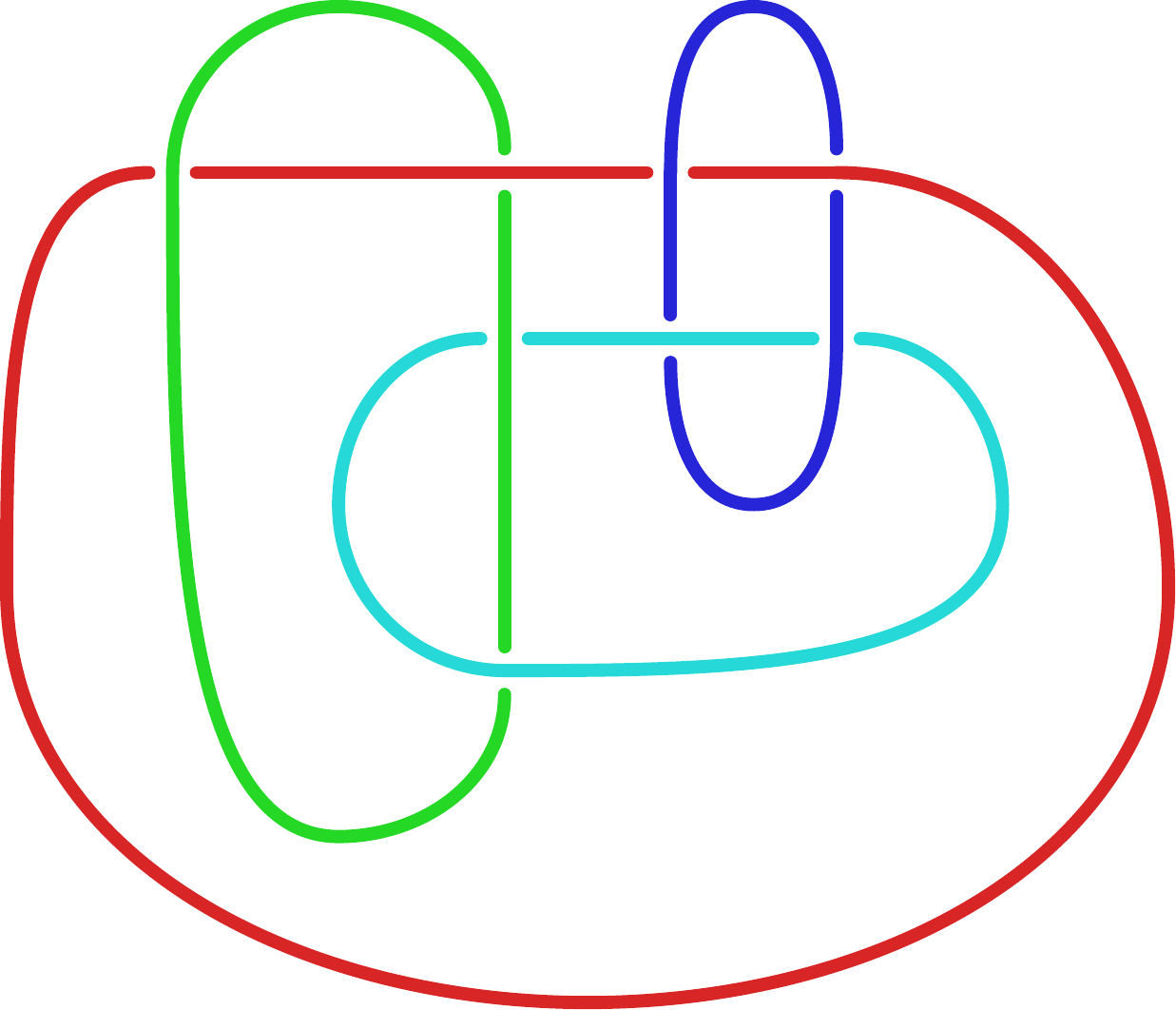}
\\
$$
\otet10_{0006}(L8a20) \qquad 
\otet10_{0042}(L10n88) \qquad 
\otet10_{0008}(L11n354) \qquad 
\otet10_{0011}(L8a21)
$$
\includegraphics[height=0.14\textheight]{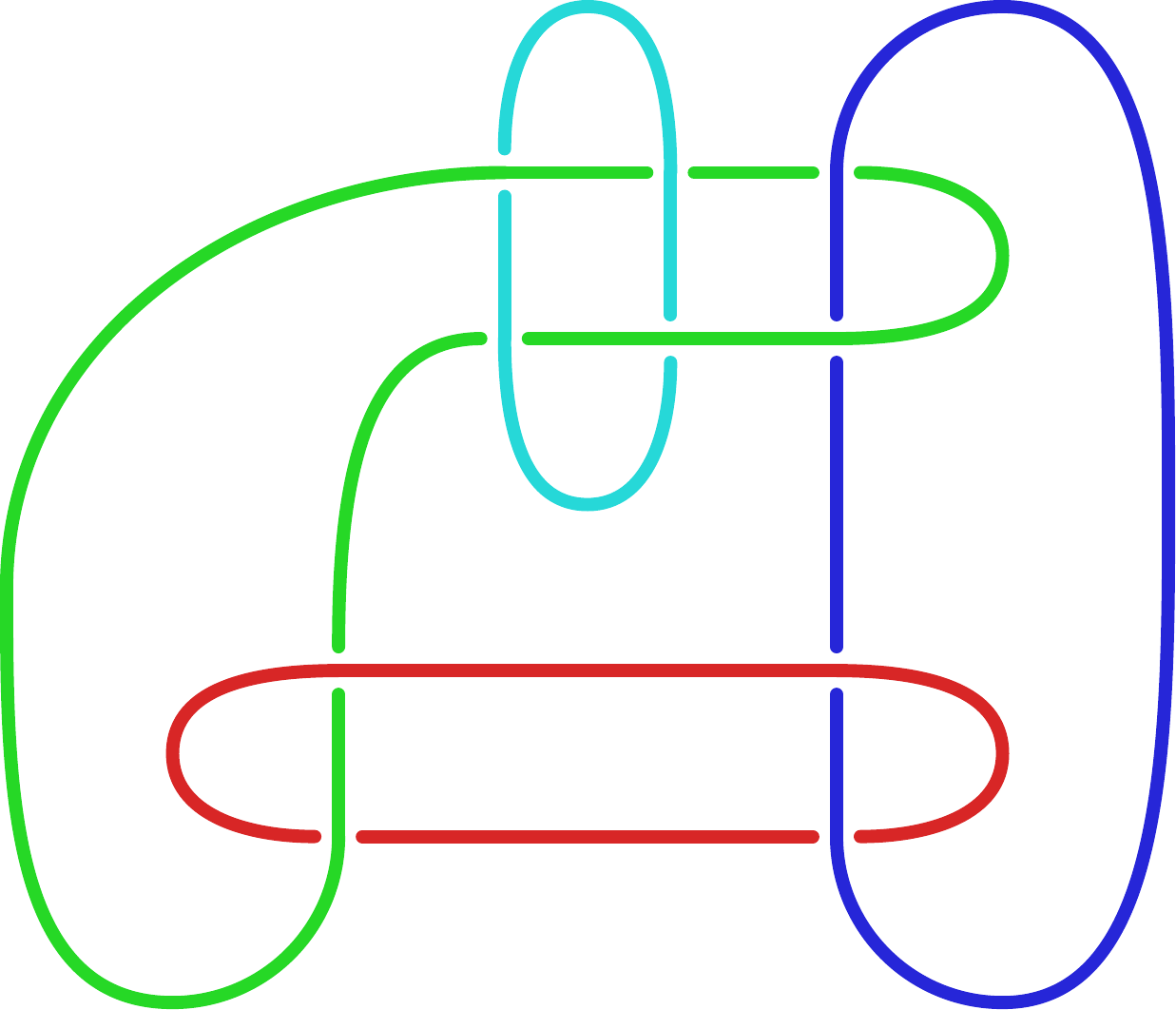}\hspace{5mm} 
\includegraphics[height=0.14\textheight]{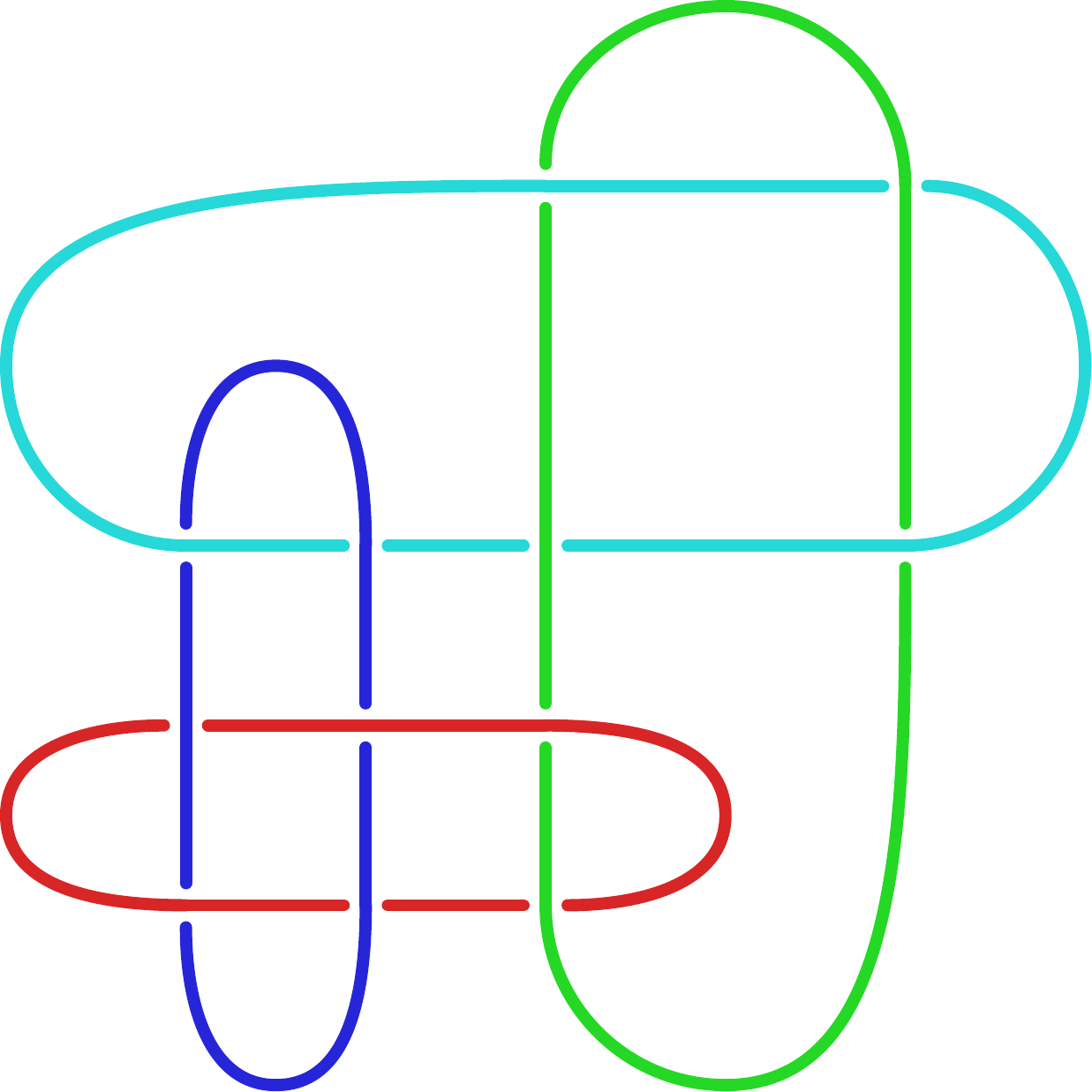}\hspace{5mm} 
\includegraphics[height=0.14\textheight]{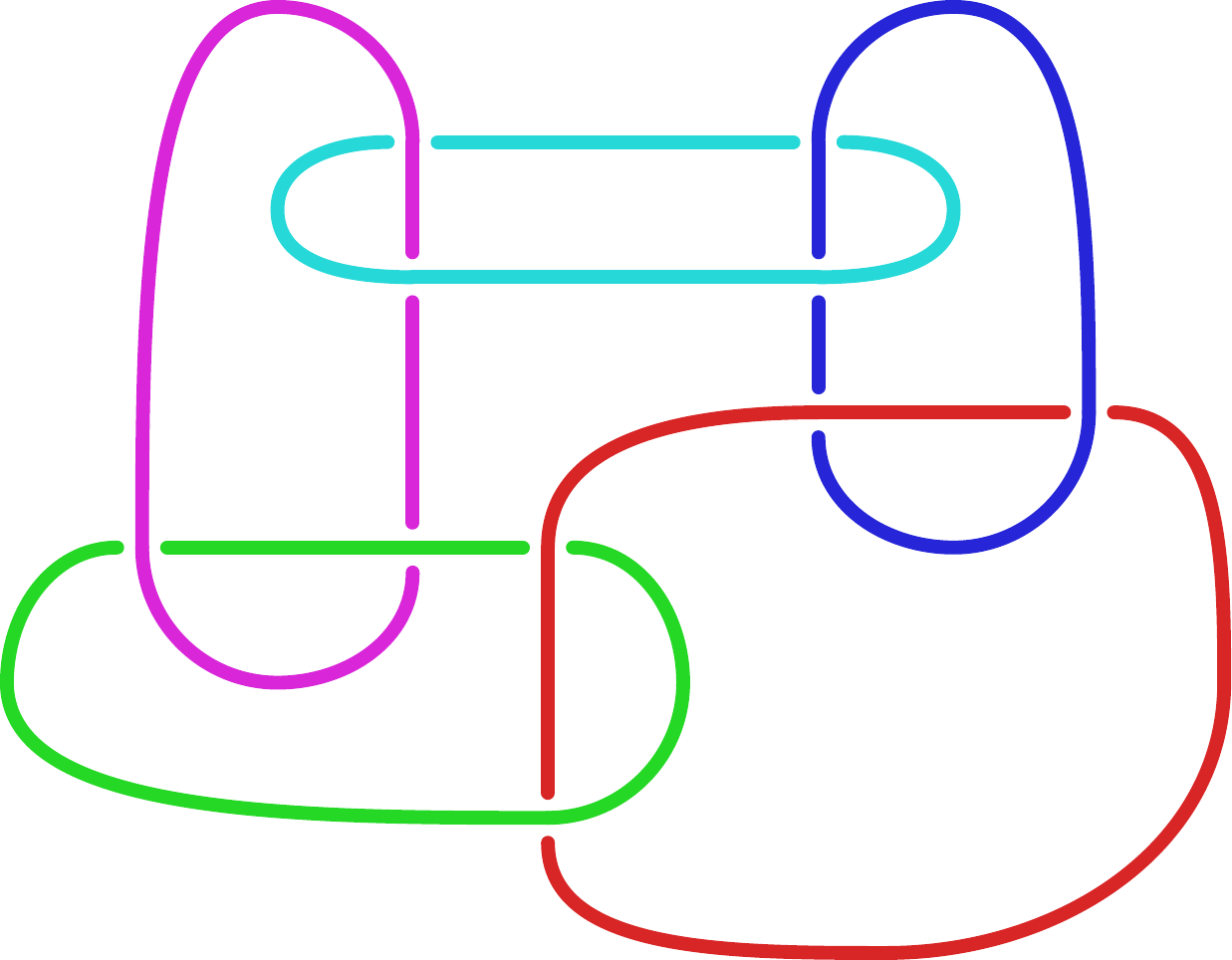}
\\
$$
\otet10_{0014}(L10n101) \qquad 
\otet10_{0028}(L12n2201) \qquad
\otet10_{0027}(L10n113)
$$
\includegraphics[height=0.15\textheight]{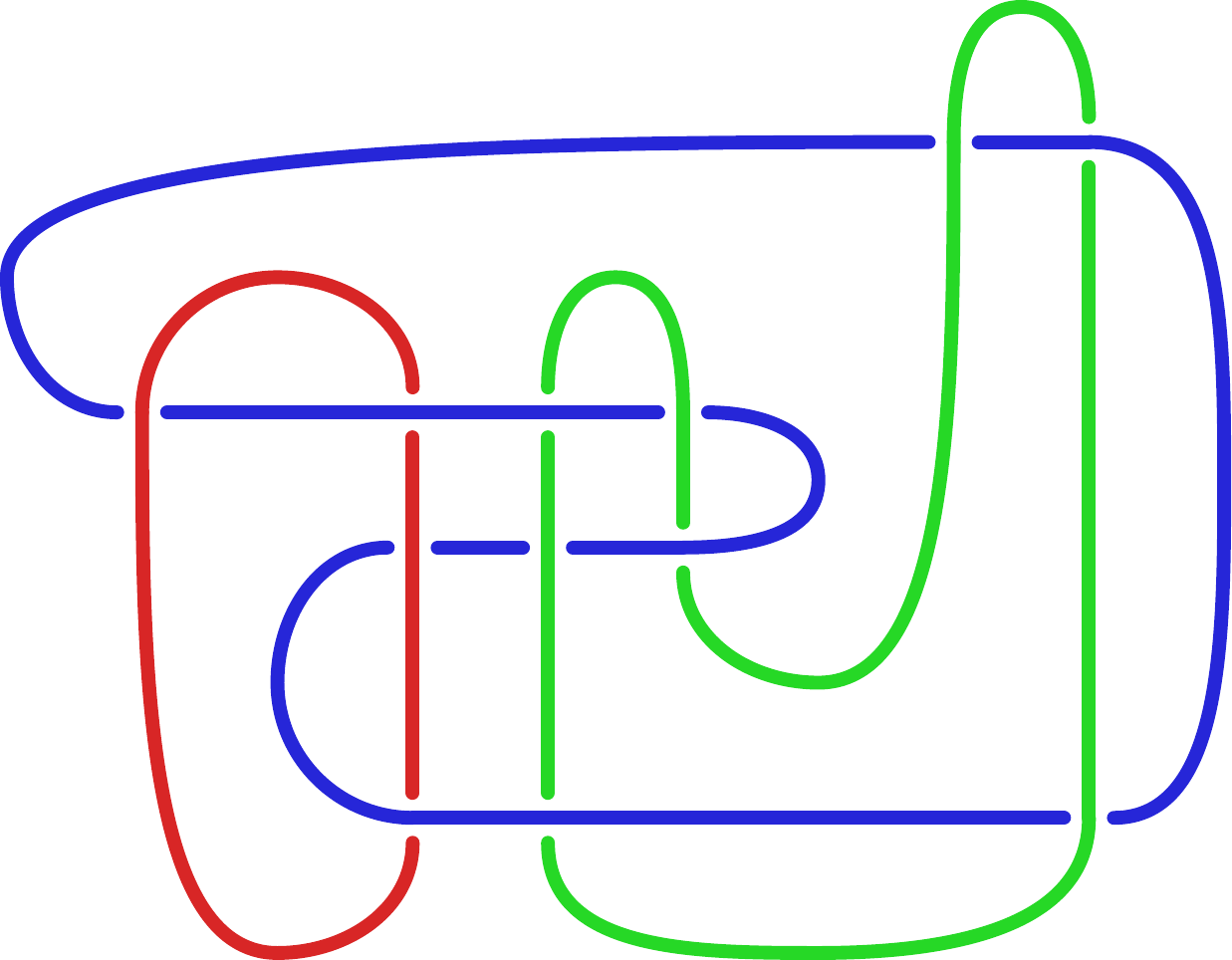}\hspace{5mm} 
\includegraphics[height=0.15\textheight]{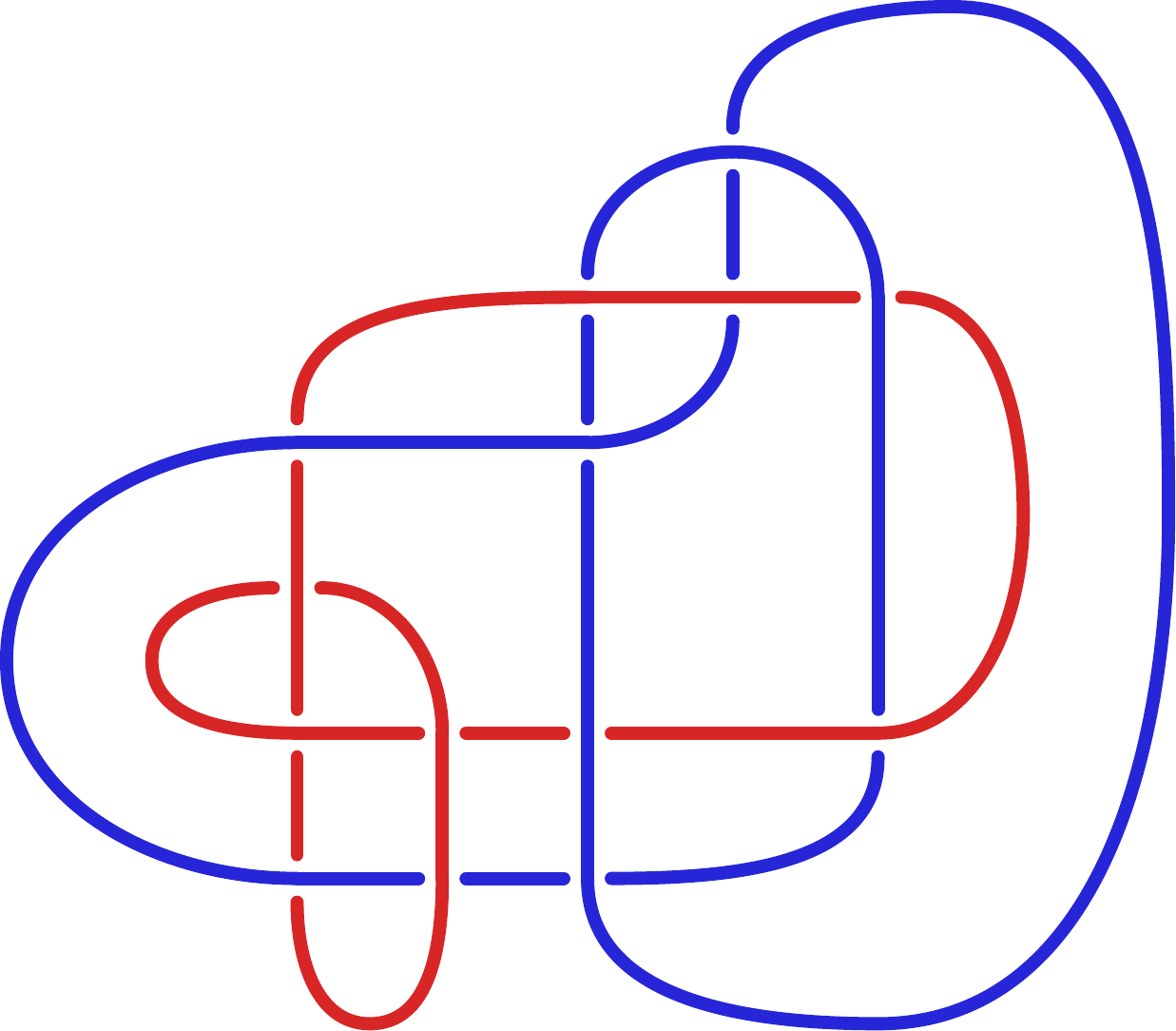} \hspace{5mm}
\includegraphics[height=0.15\textheight]{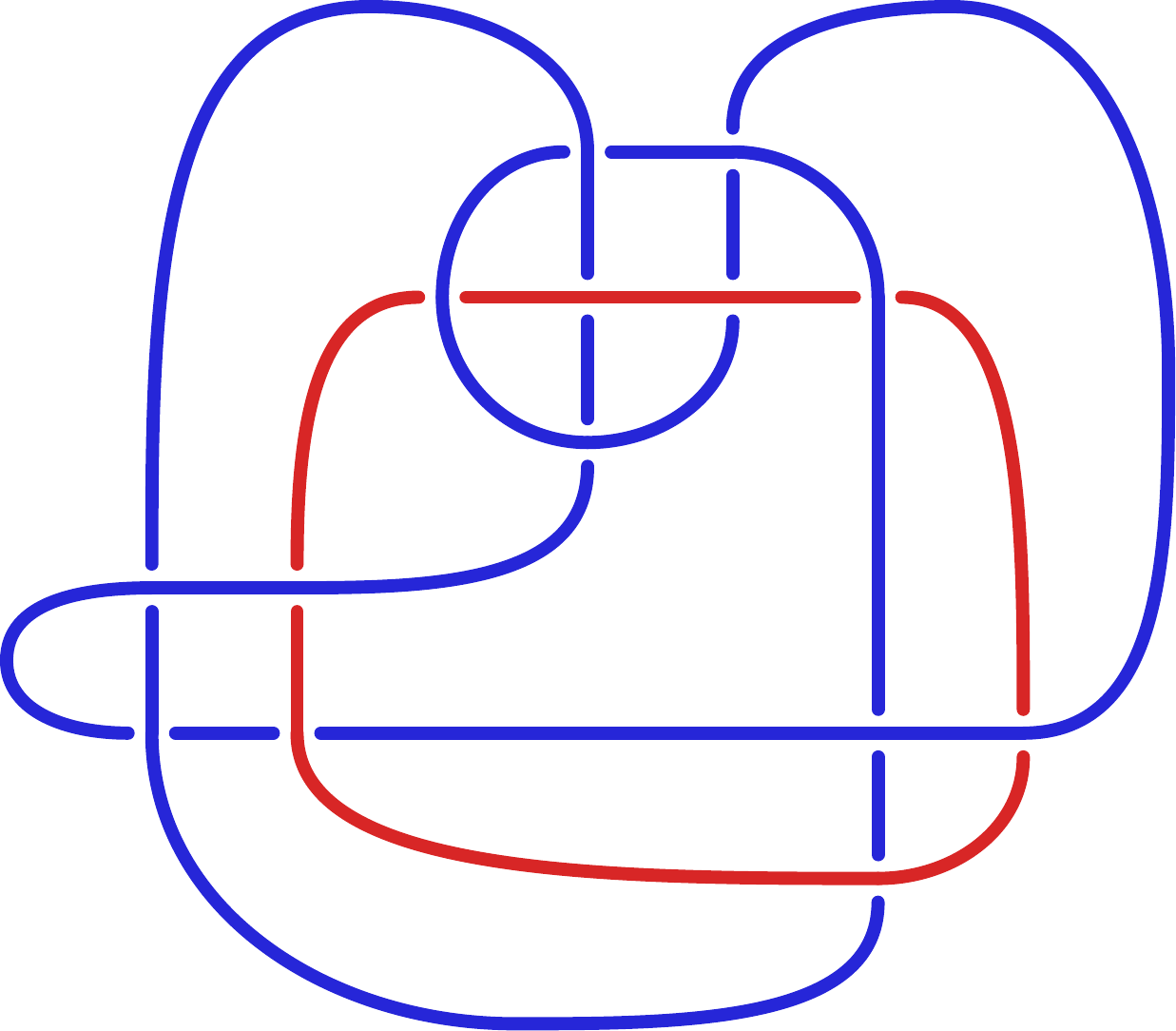}
\\
$$
\otet10_{0043}(L12n1739) \qquad 
\otet08_{0009}(L14n38547) \qquad
\otet08_{0001}(L14n24613)
$$
\includegraphics[height=0.15\textheight]{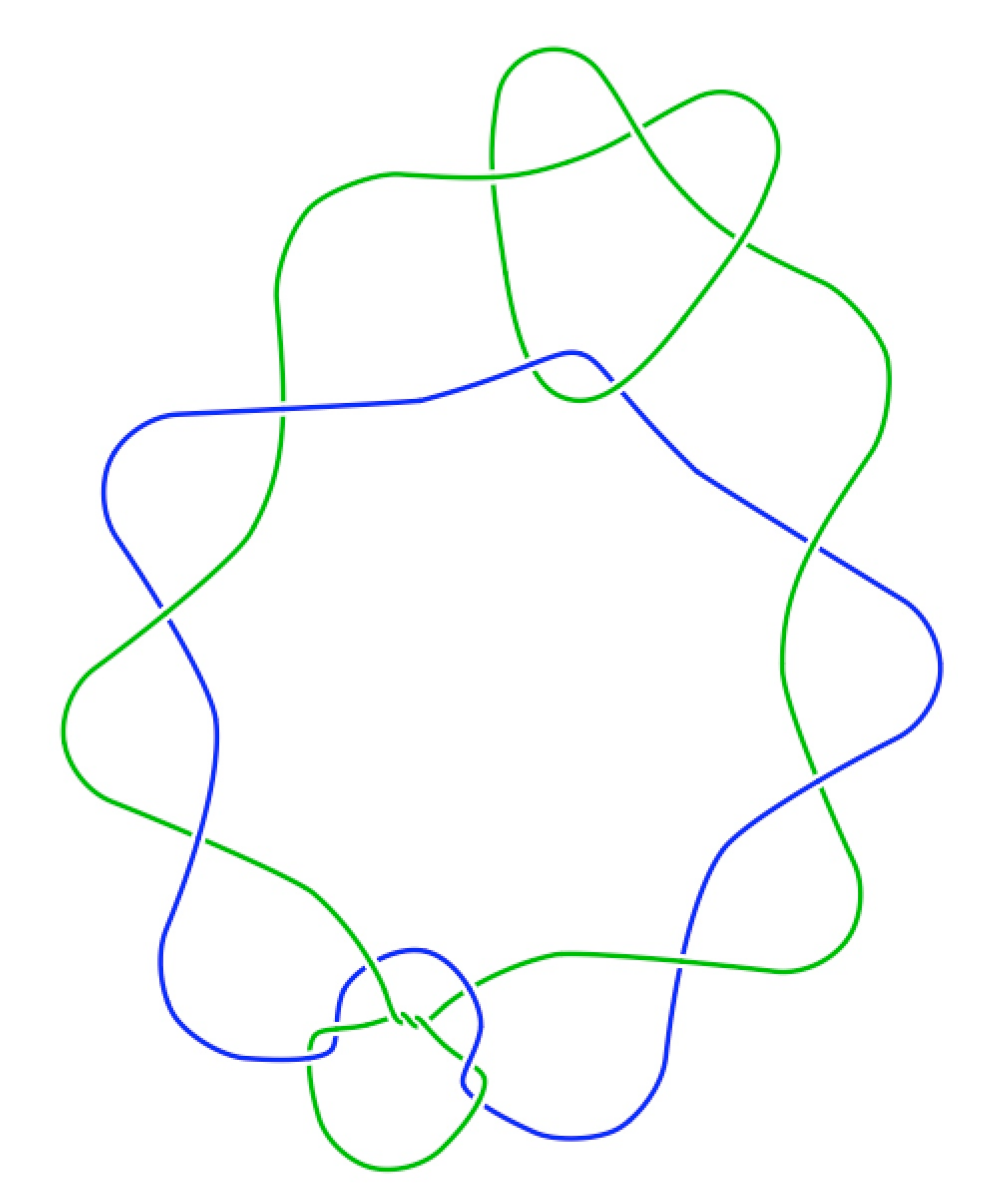}\hspace{5mm} 
\includegraphics[height=0.15\textheight]{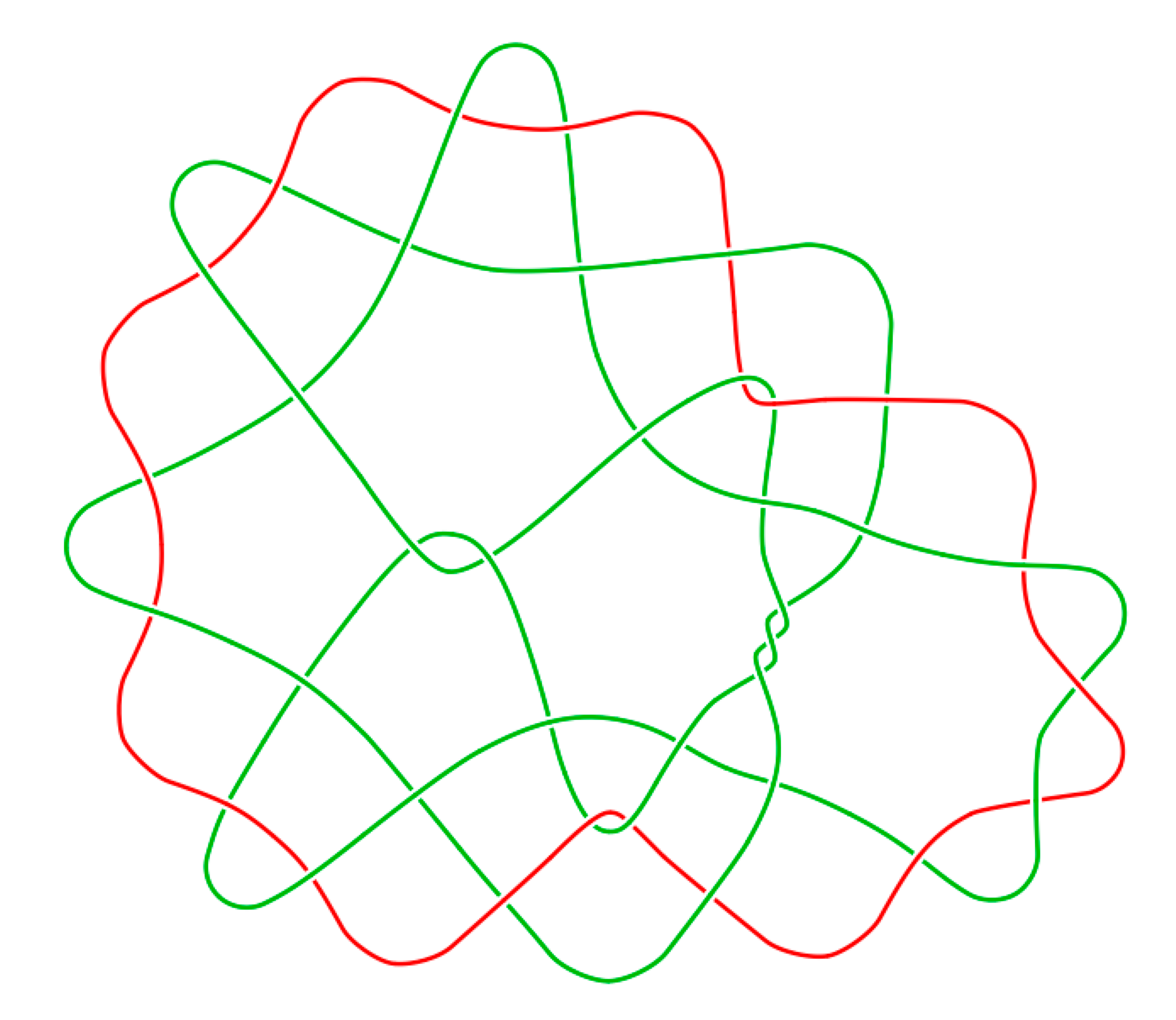} \hspace{5mm}
\includegraphics[height=0.15\textheight]{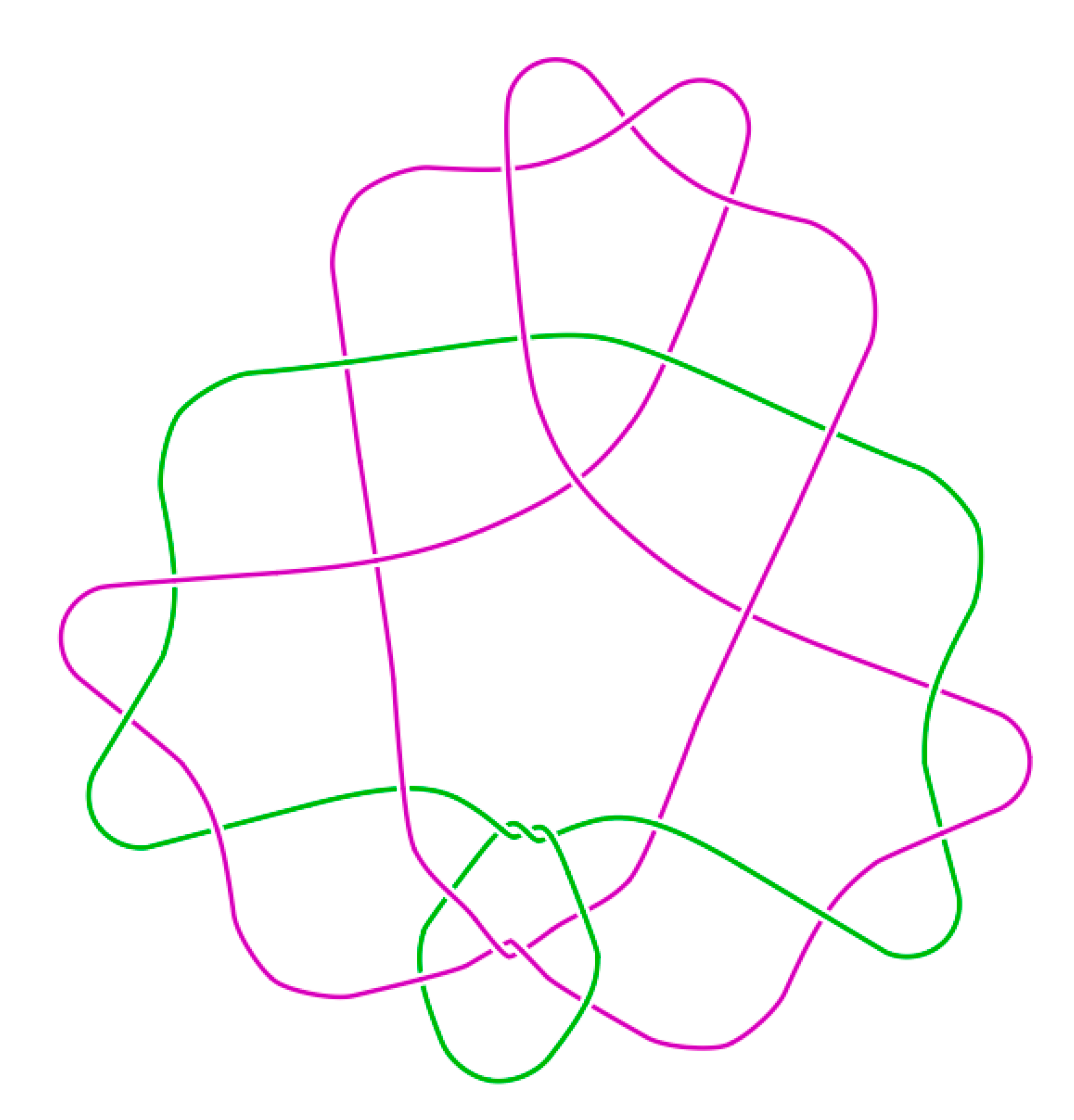}\hspace{5mm}
\includegraphics[height=0.15\textheight]{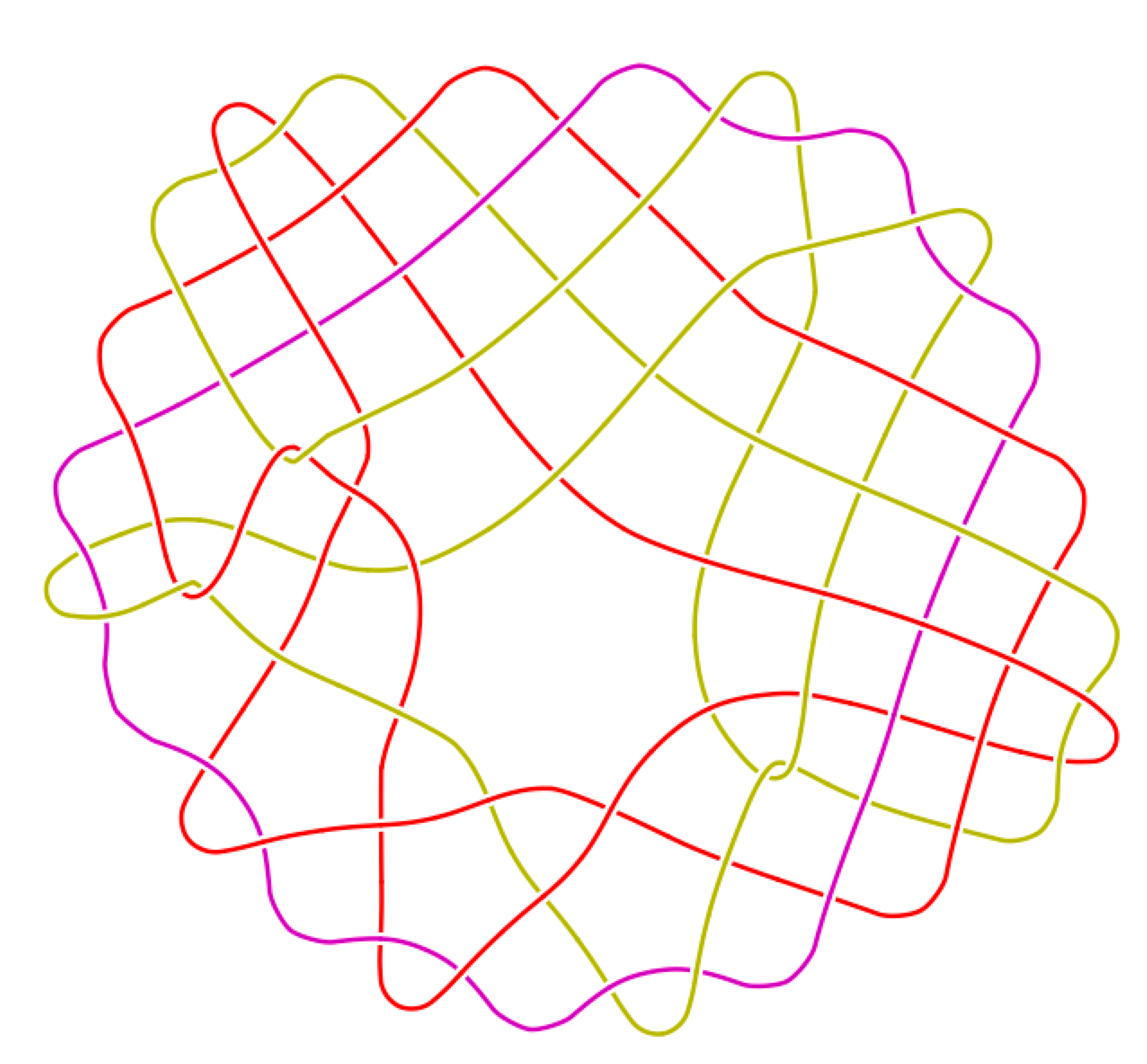}
\\
$$
\otet08_{0005} \qquad \qquad \quad
\otet10_{0007} \qquad \qquad \qquad \quad
\otet10_{0003} \qquad \qquad \quad
\otet10_{0025}
$$
\caption{The tetrahedral links with at most 10 tetrahedra.}
\label{f.tetlinks}
}
\end{figure}

\begin{figure}[!hptb] 
\centering{
\includegraphics[height=0.14\textheight]{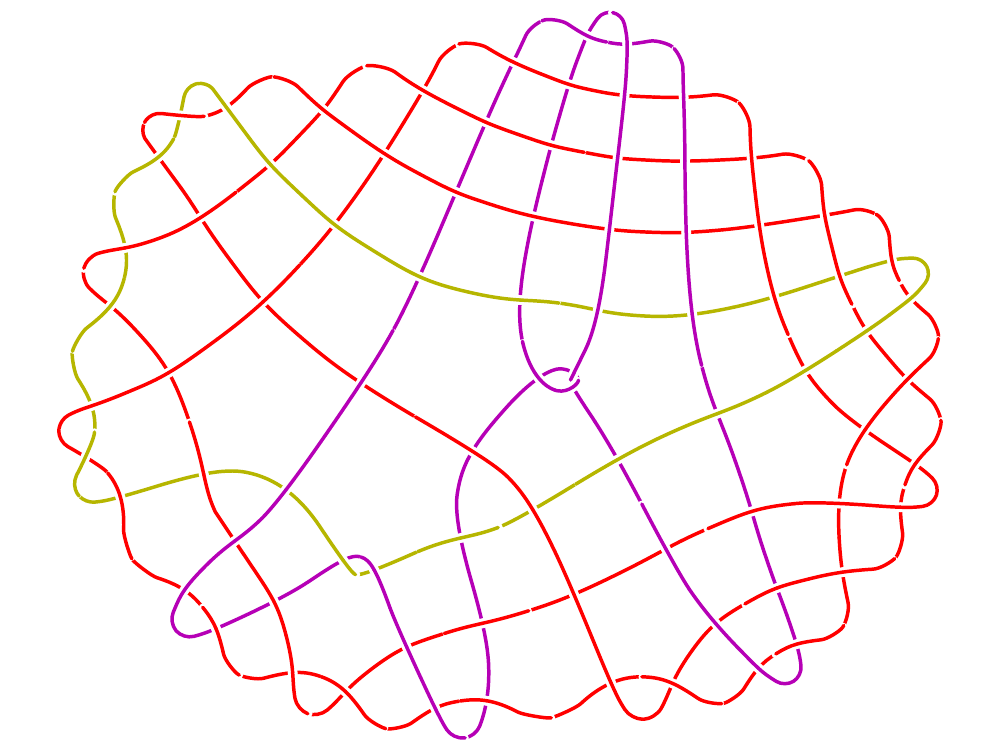}\hspace{5mm} 
\includegraphics[height=0.14\textheight]{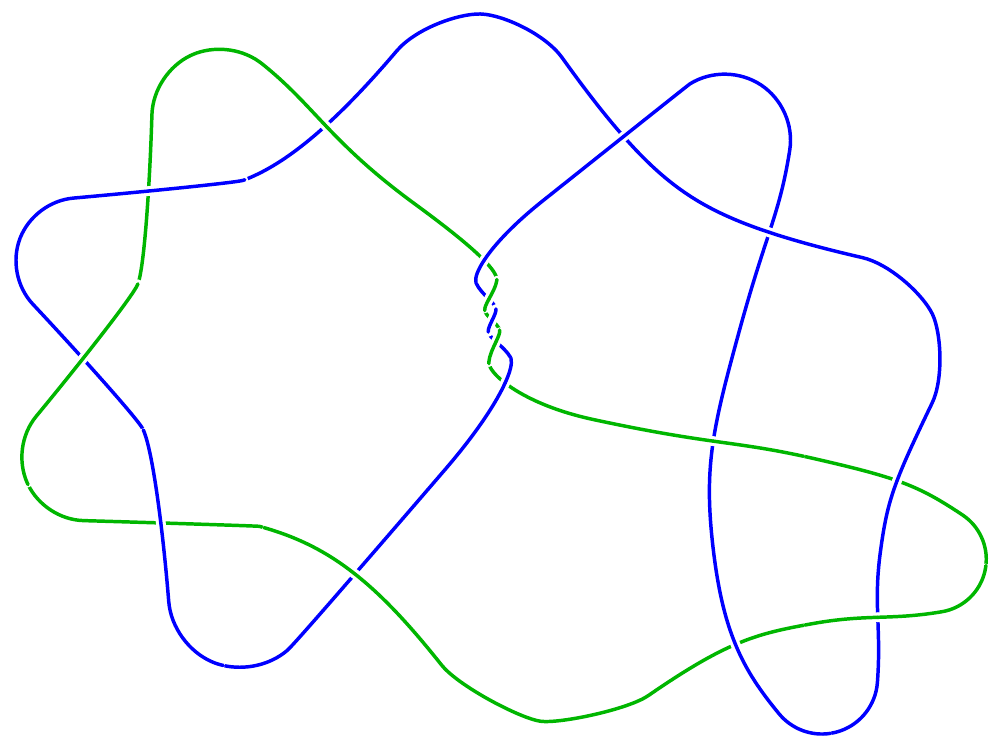}\hspace{5mm}
\\
$$
\otet12_{0001} \qquad \qquad \qquad \quad \otet12_{0005}
$$
\includegraphics[height=0.14\textheight]{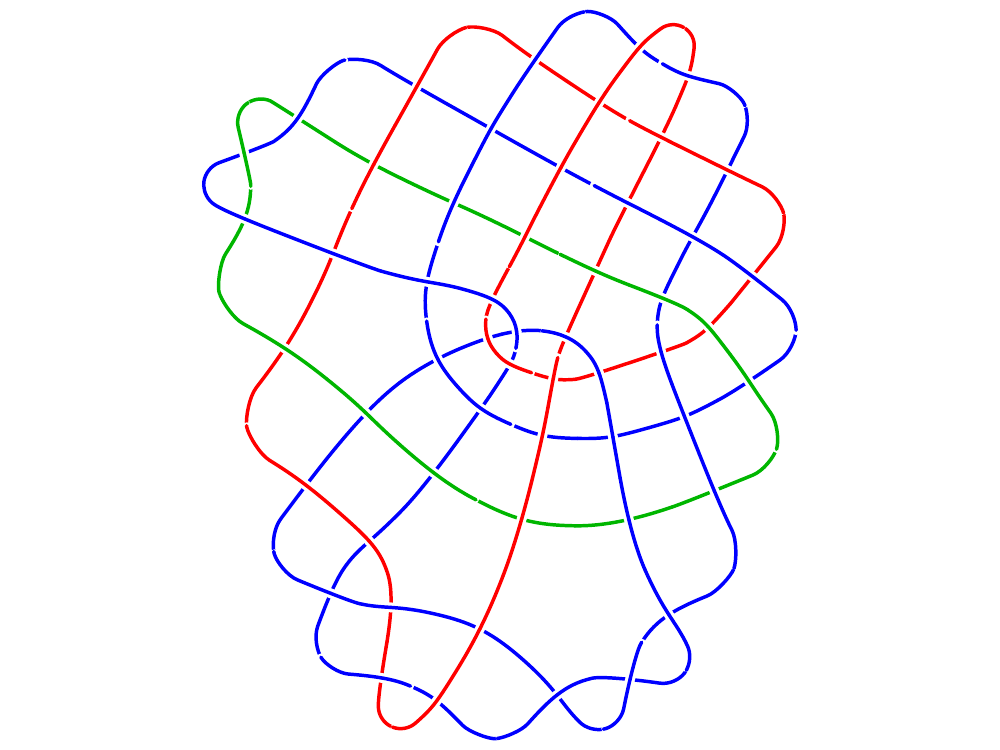}\hspace{5mm}
\includegraphics[height=0.14\textheight]{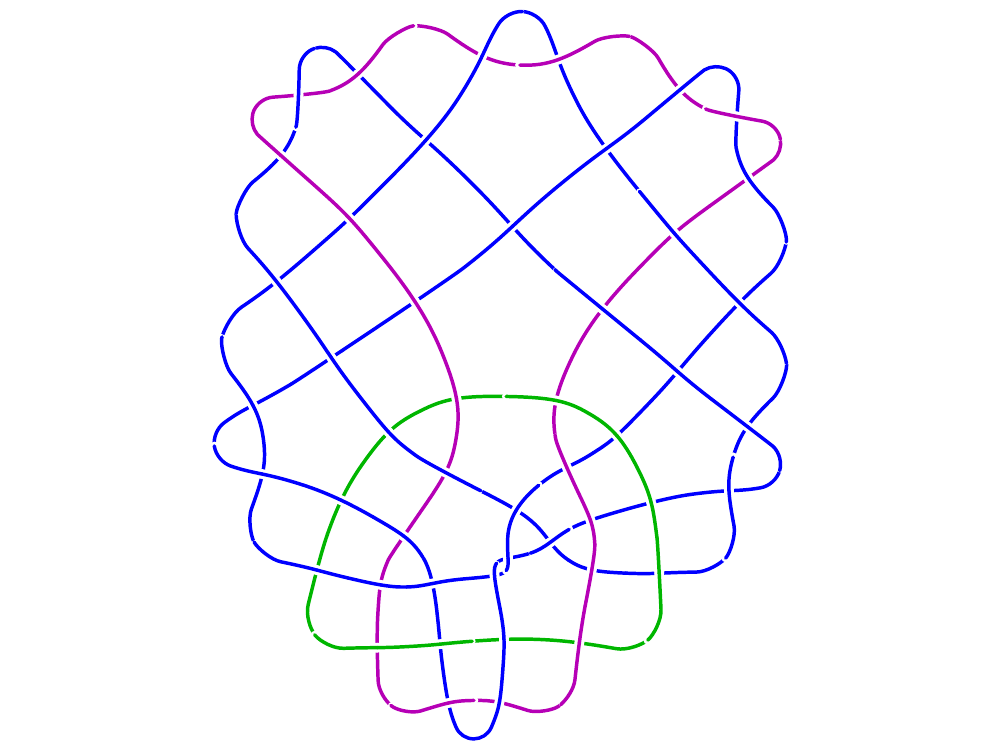}\hspace{5mm}
\\
$$
\otet12_{0006} \qquad \qquad \qquad \quad \otet12_{0010}
$$
\\
\includegraphics[height=0.14\textheight]{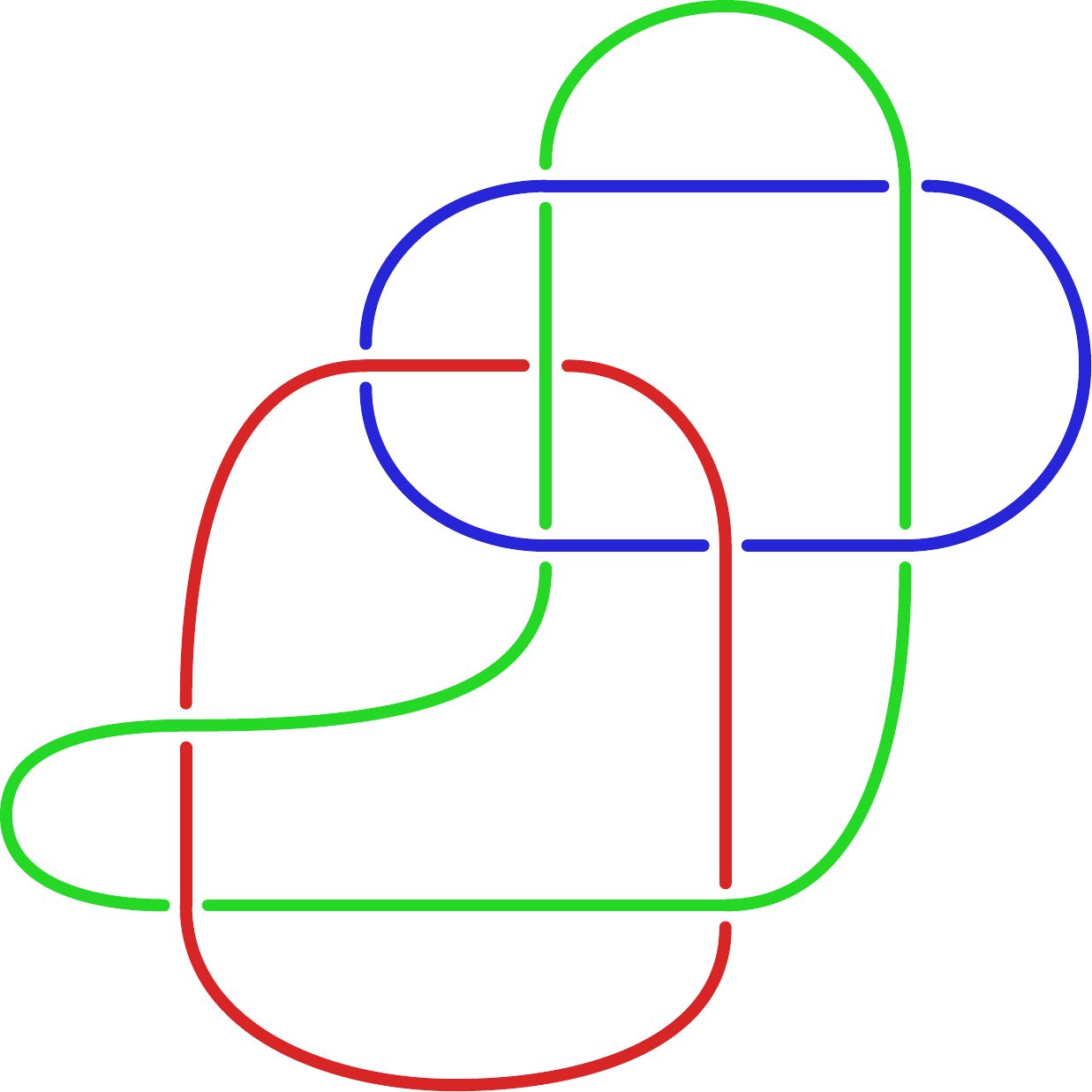}\hspace{12mm} 
\includegraphics[height=0.14\textheight]{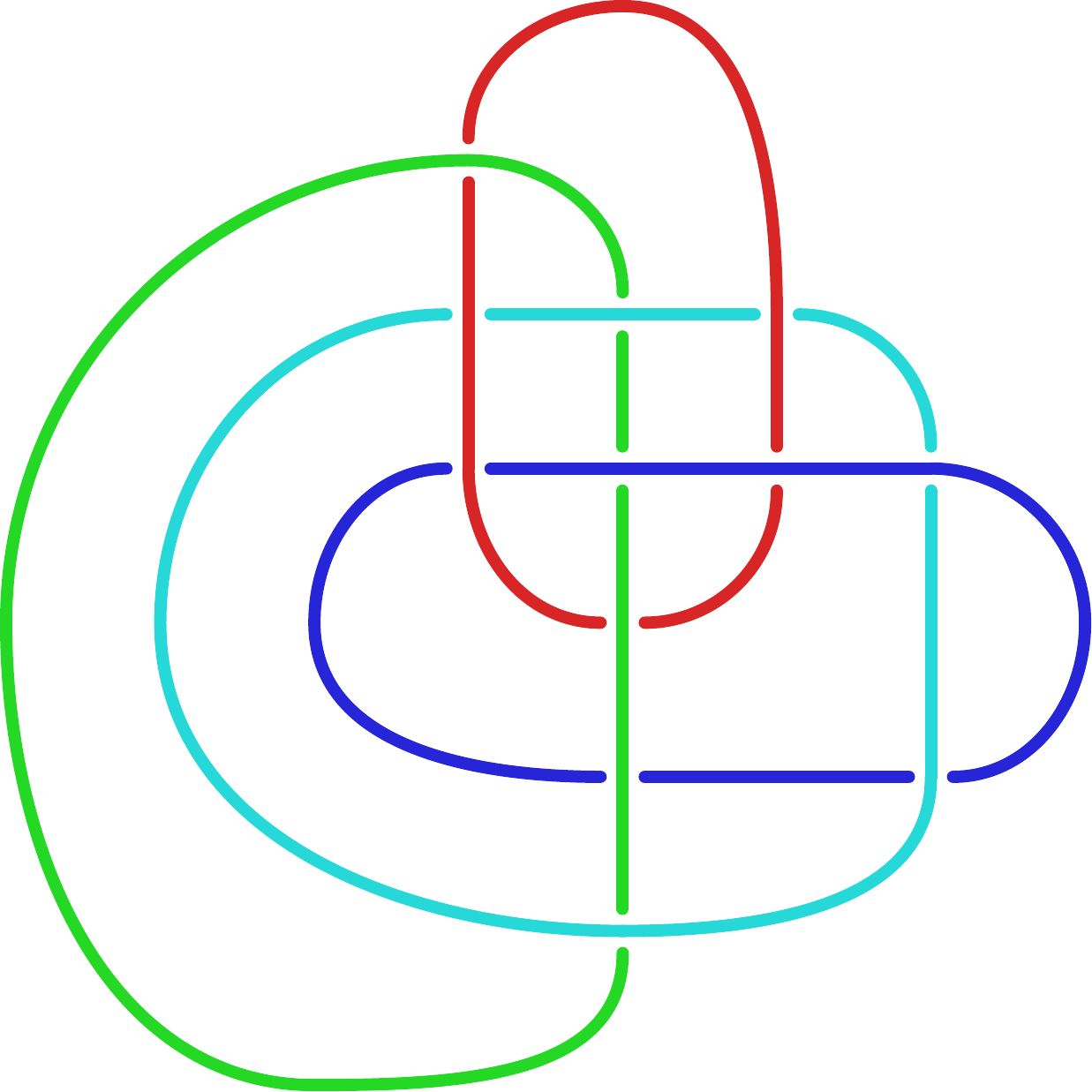}\hspace{12mm}
\includegraphics[height=0.14\textheight]{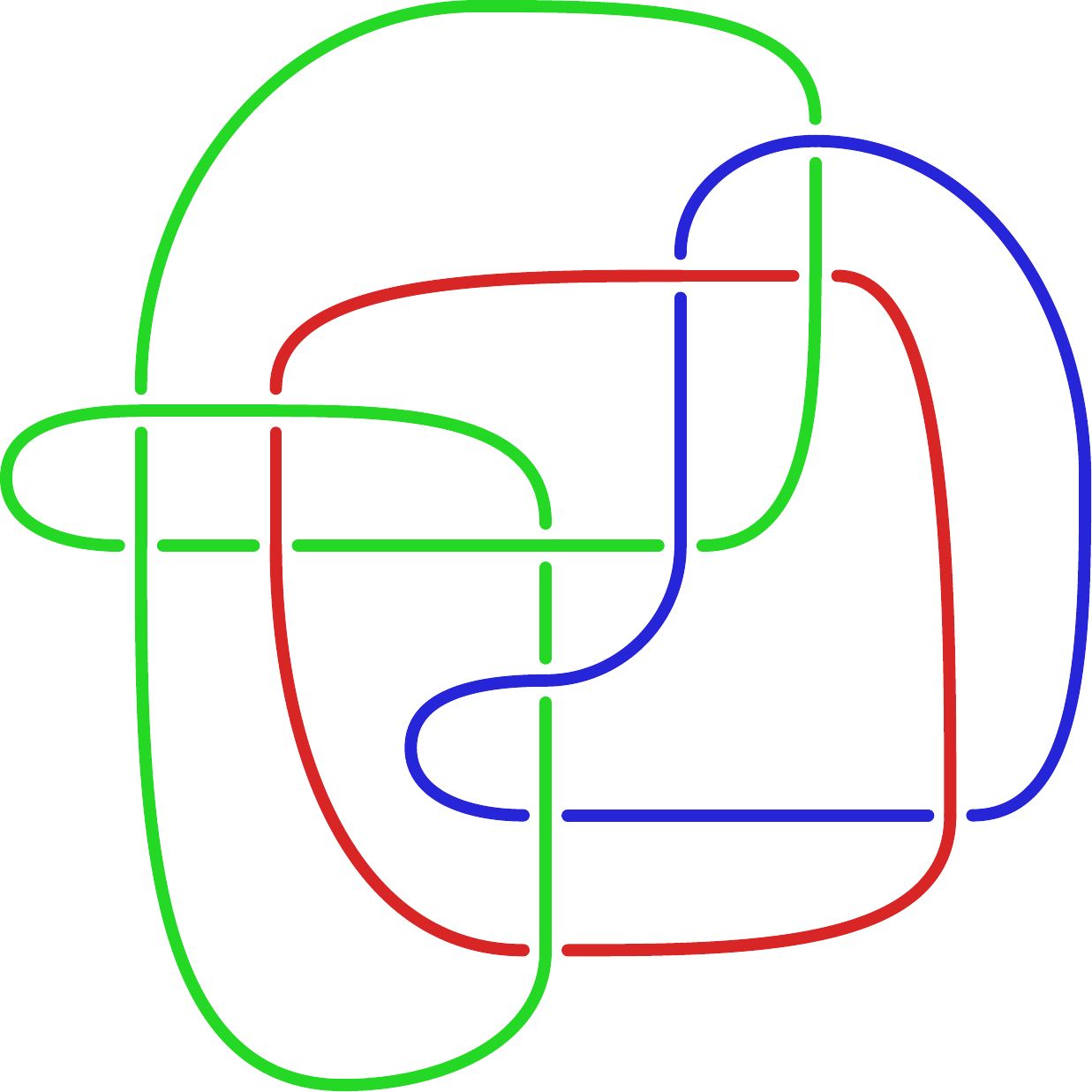}\hspace{2mm}
\\
$$
\otet12_{0007}(L10a157) \qquad
\otet12_{0009}(L12n2208) \qquad
\otet12_{0018}(L13n9382) 
$$
\caption{The tetrahedral links with 12 tetrahedra.}
\label{f.tetlinks12}
}
\end{figure}

\begin{figure}[!hptb] 
\centering{
\includegraphics[height=0.20\textheight]{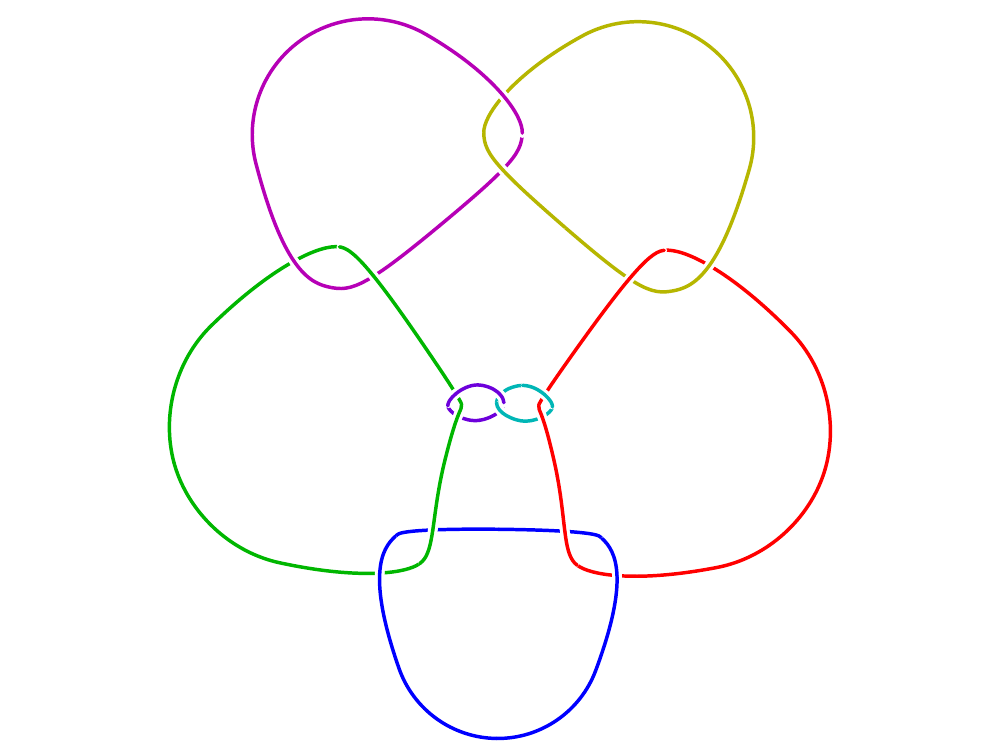}
\caption{The remarkable link $\otet20_{0570}$.} 
\label{fig.remarkable}
}
\end{figure}

\subsection*{Acknowledgment}
E.F., V.T. and A.V. were supported in part by the Ministry of Education and 
Science of the Russia (the state task number 1.1260.2014/K) and RFBR grant 
16-01-00414. S.G. was 
supported in part by a National Science Foundation grant DMS-14-06419. M. G. 
was supported in part by a National Science Foundation grant DMS-11-07452. 

We would like to thank Frank Swenton for adding the features to the
Kirby calculator \cite{kirbycalculator} necessary to draw the new 
tetrahedral links, and Cameron Gordon for completing the proof of Lemma
\ref{lem.homology.link}.

\bibliographystyle{hamsalpha}
\bibliography{biblio}

\newcommand{\etalchar}[1]{$^{#1}$}
\providecommand{\bysame}{\leavevmode\hbox to3em{\hrulefill}\thinspace}
\providecommand{\href}[2]{#2}
\providecommand{\eprint}{\begingroup \urlstyle{rm}\Url}
\begin{thebibliography}{VTF14b}

\bibitem[Ani05]{A}
S.~Anisov, \emph{Exact values of complexity for an infinite number of
  3-manifolds}, Moscow Math. J.. \textbf{5} (2005), no.~2, 305--310.

\bibitem[AR92]{AR:dodecahedral}
I.~R. Aitchison and J.~H. Rubinstein, \emph{Combinatorial cubings, cusps, and
  the dodecahedral knots}, Topology '90 ({C}olumbus, {OH}, 1990), Ohio State
  Univ. Math. Res. Inst. Publ., vol.~1, de Gruyter, Berlin, 1992, pp.~17--26.

\bibitem[BMR95]{BMR95}
B.~H. Bowditch, C.~Maclachlan, and A.~W. Reid, \emph{Arithmetic hyperbolic
  surface bundles}, Math. Ann. \textbf{302} (1995), no.~1, 31--60.

\bibitem[BP14]{Burton:edges}
Benjamin~A. Burton and William Pettersson, \emph{An edge-based framework for
  enumerating 3-manifold triangulations}, 2014, \eprint{arXiv:1412.2169},
  Preprint.

\bibitem[Bre97]{bredon:top_and_geo}
Glen~E. Bredon, \emph{Topology and geometry}, Graduate Texts in Mathematics,
  vol. 139, Springer-Verlag, New York, 1997, Corrected third printing of the
  1993 original.

\bibitem[Bur]{Regina}
Benjamin Burton, \emph{{R}egina, software for $3$-manifold topology and normal
  surface theory}, \url{http://regina.sourceforge.net} (30/01/2015).

\bibitem[Bur11]{burton:encode}
Benjamin~A. Burton, \emph{The {P}achner graph and the simplification of
  3-sphere triangulations}, Computational geometry ({SCG}'11), ACM, New York,
  2011, \eprint{arXiv:1110.6080}, pp.~153--162.

\bibitem[Bur14]{burton:x101}
\bysame, \emph{A duplicate pair in the {S}nap{P}ea census}, Exp. Math.
  \textbf{23} (2014), no.~2, 170--173.

\bibitem[BZ85]{BZ}
Gerhard Burde and Heiner Zieschang, \emph{Knots}, de Gruyter Studies in
  Mathematics, vol.~5, Walter de Gruyter \& Co., Berlin, 1985.

\bibitem[CDW]{SnapPy}
Marc Culler, Nathan~M. Dunfield, and Jeffrey~R. Weeks, \emph{Snap{P}y, a
  computer program for studying the topology of $3$-manifolds}, Available at
  \url{http://snappy.computop.org} (30/01/2015).

\bibitem[CHW99]{CHW}
Patrick~J. Callahan, Martin~V. Hildebrand, and Jeffrey~R. Weeks, \emph{A census
  of cusped hyperbolic {$3$}-manifolds}, Math. Comp. \textbf{68} (1999),
  no.~225, 321--332, With microfiche supplement.

\bibitem[DHL14]{DHL}
Nathan~M. Dunfield, Neil~R. Hoffman, and Joan~E. Licata, \emph{Asymmetric
  hyperbolic {L}-spaces, {H}eegaard genus, and {D}ehn filling}, 2014,
  \eprint{arXiv:1407.7827}, Preprint.

\bibitem[EP88]{EP}
D.~B.~A. Epstein and R.~C. Penner, \emph{Euclidean decompositions of noncompact
  hyperbolic manifolds}, J. Differential Geom. \textbf{27} (1988), no.~1,
  67--80.

\bibitem[GHH08]{goodmanHeardHodgson:CommensuratorsOfCuspedHyperbolicManifolds}
Oliver Goodman, Damian Heard, and Craig Hodgson, \emph{Commensurators of cusped
  hyperbolic manifolds}, Experiment. Math. \textbf{17} (2008), no.~3, 283--306.

\bibitem[GL89]{Gordon}
C.~McA. Gordon and J.~Luecke, \emph{Knots are determined by their complements},
  J. Amer. Math. Soc. \textbf{2} (1989), no.~2, 371--415.

\bibitem[Goe]{Goerner:tetcensus}
Matthias Goerner, \url{http://unhyperbolic.org/tetrahedralCensus/}.

\bibitem[Goe15]{Goerner}
\bysame, \emph{Regular {T}essellation {L}ink {C}omplements}, Experiment. Math.
  \textbf{24} (2015), no.~2, 225--246, \eprint{http://arxiv.org/abs/1406.2827}.

\bibitem[HIK{\etalchar{+}}13]{hikmot}
N.~Hoffman, K.~Ichihara, M.~Kashiwagi, H.~Masai, S.~Oishi, and A.~Takayasu,
  \emph{Verified computations for hyperbolic 3-manifolds}, 2013,
  \eprint{arXiv:1310.3410},
  \url{http://www.oishi.info.waseda.ac.jp/~takayasu/hikmot}.

\bibitem[IM]{fefgen}
K.~Ichihara and H.~Masai, \emph{{\texttt{fef\_gen.py}}},
  \url{http://www.math.chs.nihon-u.ac.jp/~ichihara/ExcAlt/}.

\bibitem[Lan02]{Lang}
Serge Lang, \emph{Algebra}, third ed., Graduate Texts in Mathematics, vol. 211,
  Springer-Verlag, New York, 2002.

\bibitem[LST08]{LuoSchleimerTillmann:virtualGeometricTrig}
Feng Luo, Saul Schleimer, and Stephan Tillmann, \emph{Geodesic ideal
  triangulations exist virtually}, Proc. Amer. Math. Soc. \textbf{136} (2008),
  no.~7, 2625--2630.

\bibitem[Mat03]{Matveev}
Sergei Matveev, \emph{Algorithmic topology and classification of 3-manifolds},
  Algorithms and Computation in Mathematics, vol.~9, Springer-Verlag, Berlin,
  2003.

\bibitem[MP06]{MP:magicmfdfilling}
Bruno Martelli and Carlo Petronio, \emph{Dehn filling of the ``magic''
  3-manifold}, Comm. Anal. Geom. \textbf{14} (2006), no.~5, 969--1026.

\bibitem[MPR14]{MPR:fivechainfilling}
Bruno Martelli, Carlo Petronio, and Fionntan Roukema, \emph{Exceptional {D}ehn
  surgery on the minimally twisted five-chain link}, Comm. Anal. Geom.
  \textbf{22} (2014), no.~4, 689--735.

\bibitem[MR03]{MR}
Colin Maclachlan and Alan~W. Reid, \emph{The arithmetic of hyperbolic
  3-manifolds}, Graduate Texts in Mathematics, vol. 219, Springer-Verlag, New
  York, 2003.

\bibitem[NR92a]{neumannReid:Topology90Arithmetic}
Walter~D. Neumann and Alan~W. Reid, \emph{Arithmetic of hyperbolic manifolds},
  Topology '90 ({C}olumbus, {OH}, 1990), Ohio State Univ. Math. Res. Inst.
  Publ., vol.~1, de Gruyter, Berlin, 1992, pp.~273--310.

\bibitem[NR92b]{neumannReid:SmallVolOrbs}
\bysame, \emph{Notes on {A}dams' small volume orbifolds}, Topology '90
  ({C}olumbus, {OH}, 1990), Ohio State Univ. Math. Res. Inst. Publ., vol.~1, de
  Gruyter, Berlin, 1992, pp.~311--314.

\bibitem[PW00]{PetronioWeeks:partiallyFlatTrig}
Carlo Petronio and Jeffrey~R. Weeks, \emph{Partially flat ideal triangulations
  of cusped hyperbolic 3-manifolds}, Osaka J. Math. \textbf{37} (2000), no.~2,
  453--466.

\bibitem[RA01]{RV}
Marco Reni and Vesnin Andrei, \emph{Hidden symmetries of cyclic branched
  coverings of 2-bridge knots}, Rend. Istit. Mat. Univ. Trieste \textbf{XXXII}
  (2001), 289--304.

\bibitem[Rei91]{reid:arithmeticity}
Alan~W. Reid, \emph{Arithmeticity of knot complements}, J. London Math. Soc.
  (2) \textbf{43} (1991), no.~1, 171--184.

\bibitem[Swe]{kirbycalculator}
Frank Swenton,
  \url{http://community.middlebury.edu/~mathanimations/kirbycalculator/}.

\bibitem[VMF11]{VMF}
A.~Yu. Vesnin, S.~V. Matveev, and E.~A. Fominykh, \emph{Complexity of
  three-dimensional manifolds: exact values and estimates (in russian)}, Sib.
  \`Elektron. Mat. Izv. \textbf{8} (2011), 341--364.

\bibitem[VTF14a]{VTF1}
A.~Yu. Vesnin, V.~V. Tarkaev, and E.~A. Fominykh, \emph{On the complexity of
  three-dimensional cusped hyperbolic manifolds}, Doklady Math. \textbf{89}
  (2014), no.~3, 267--270.

\bibitem[VTF14b]{VTF2}
\bysame, \emph{Three-dimensional hyperbolic manifolds with cusps of complexity
  10 having maximal volume (in russian)}, Trudy Instituta Matematiki i
  Mekhaniki \textbf{20} (2014), no.~2, 74--87.

\bibitem[Wal11]{Walsh}
Genevieve~S. Walsh, \emph{Orbifolds and commensurability}, Interactions between
  hyperbolic geometry, quantum topology and number theory, Contemp. Math., vol.
  541, Amer. Math. Soc., Providence, RI, 2011, pp.~221--231.

\bibitem[Wee93]{Weeks}
Jeffrey~R. Weeks, \emph{Convex hulls and isometries of cusped hyperbolic
  {$3$}-manifolds}, Topology Appl. \textbf{52} (1993), no.~2, 127--149.

\end{thebibliography}

\end{document}